%% file: Anticolimits.tex
\documentclass[a4paper]{article}

\input{preamble}

\title{The theory and applications of anticolimits}
\author{%
Calin Tataru
\thanks{University of Cambridge, \email{calin.tataru@cl.cam.ac.uk}}
\and
Jamie Vicary
\thanks{University of Cambridge, \email{jamie.vicary@cl.cam.ac.uk}}
}

\begin{document}

\maketitle

\input{abstract}

\input{1-introduction}

\input{2-anticolimits}

\input{3-zigzags}

\input{4-anticontraction}

\bibliographystyle{plainurl}
\bibliography{Anticolimits}

\end{document}

%% file: preamble.tex
\usepackage{amsmath}
\usepackage{amssymb}
\usepackage{amsthm}
\usepackage{bussproofs}
\usepackage{hyperref}
\usepackage{cleveref}
\usepackage{mathtools}
\usepackage{tikz}
\usepackage{tikz-cd}

\newcommand\ord[1]{\mathbf{#1}}

\newcommand\pos[1]{\mathbf{#1}}

\newcommand\cat[1]{\mathbf{#1}}

\newcommand\Cat{\mathbf{Cat}}
\newcommand\Pos{\mathbf{Pos}}
\newcommand\Pre{\mathbf{Pre}}
\newcommand\Set{\mathbf{Set}}
\newcommand\Ord{\mathbf{\Delta}}
\newcommand\Int{\boldsymbol{\nabla}}

\newcommand\email[1]{\href{mailto:#1}{#1}}

\DeclarePairedDelimiter\set{\{}{\}}
\DeclarePairedDelimiter\abs{\lvert}{\rvert}

\DeclareMathOperator\id{id}
\DeclareMathOperator\colim{colim}

\DeclareMathOperator\upper{\uparrow}
\DeclareMathOperator\Upper{\Uparrow}

\DeclareMathOperator\Ext{\mathsf{Ext}}
\DeclareMathOperator\Acl{\mathsf{Acl}}
\DeclareMathOperator\Acc{\mathsf{Acc}}

\DeclareMathOperator\reg{\mathsf{r}}
\DeclareMathOperator\sing{\mathsf{s}}
\DeclareMathOperator\Reg{\mathsf{R}}
\DeclareMathOperator\Zig{\mathsf{Zig}}
\DeclareMathOperator\Exp{\mathsf{Exp}}

\DeclareMathOperator\Sing{\mathsf{Sing}}
\DeclareMathOperator\Relab{\mathsf{Relab}}

\newtheorem{theorem}{Theorem}

\newtheorem{corollary}[theorem]{Corollary}
\newtheorem{lemma}[theorem]{Lemma}
\newtheorem{proposition}[theorem]{Proposition}

\theoremstyle{definition}
\newtheorem{definition}[theorem]{Definition}
\newtheorem{example}[theorem]{Example}

%% file: abstract.tex
\begin{abstract}
Colimits are a fundamental construction in category theory.
They provide a way to construct new objects by gluing together existing objects that are related in some way.
We introduce a complementary notion of \emph{anticolimits}, which provide a way to decompose an object into a colimit of other objects.
While anticolimits are not unique in general, we establish that in the presence of pullbacks, there is a ``canonical'' anticolimit which characterises the existence of other anticolimits.
We also provide convenient techniques for computing anticolimits, by changing either the shape or ambient category.

The main motivation for this work is the development of a new method, known as \emph{anticontraction}, for constructing homotopies in the proof assistant \textsc{homotopy.io} for finitely presented $n$-categories.
Anticontraction complements the existing contraction method and facilitates the construction of homotopies increasing the complexity of a term, enhancing the usability of the proof assistant.
For example, it simplifies the naturality move and third Reidemeister move.
\end{abstract}

%% file: 1-introduction.tex
\section{Introduction} \label{sec:introduction}

Colimits are an important tool in computer science for gluing things together.
A notable example is \emph{double pushout graph rewriting}~\cite{EPS73,EEHP04} which uses pushouts to glue the right-hand side of a rewrite in the place of a matching subgraph.
Colimits have also been applied to merge patches in a version control system~\cite{AMLH16}, compose open Petri nets~\cite{BGMS21}, compose systems in reinforcement learning~\cite{BST22}, give a formal semantics of visual programming~\cite{Gib02}, model multiple inheritance in programming languages~\cite{LP90}, and compose specifications in formal software development~\cite{Smi06}.

In each of these examples, one can ask the following question: when does a particular object {arise} as a colimit of some given shape?
To answer this, we introduce the notion of \emph{anticolimits}.
This construction allow us to take an object, equipped with a family of incoming morphisms, and express it as the colimit of a poset-shaped diagram.
Anticolimits may not exist in general, and unlike colimits, they are not guaranteed to be unique.
We may also consider \emph{anticocones}, a similar construction where we do not require the universal property.

To illustrate this, consider the case of pushouts.
Recall that, for a pair of morphisms $p : D \to A$ and $q : D \to B$, their pushout is a pair of morphisms $f : A \to C$ and $g : B \to C$ such that the resulting square commutes and satisfies a universal property:
\[
\begin{tikzcd}
& C \ar[dd, phantom, "\llcorner" rotate=45, very near start] & \\
A \urar[dashed, "f"] && B \ular[dashed, "g"'] \\
& D \ular["p"] \urar["q"'] & 
\end{tikzcd}
\]
Here, the solid arrows are the {input} to the pushout and the dashed arrows represent the {output}.
The universal property implies that pushouts are unique up to isomorphism.

The corresponding notion of anticolimit is called an \emph{antipushout}, and is defined as follows.
Given a pair of morphisms $f : A \to C$ and $g : B \to C$, an antipushout is given by morphisms $p : D \to A$ and $q : D \to B$ such that the following square is a pushout:
\[
\begin{tikzcd}
& C \ar[dd, phantom, "\llcorner" rotate=45, very near start] & \\
A \urar["f"] && B \ular["g"'] \\
& D \ular[dashed, "p"] \urar[dashed, "q"'] & 
\end{tikzcd}
\]
If a pushout is understood as gluing two objects together along a common part to form a new object, then an antipushout gives us a way to decompose an object into such a gluing arrangement.

\paragraph{General results}
In~\Cref{sec:anticolimits} we study the abstract theory of anticolimits.
We show that in the presence of pullbacks, a canonical anticocone can be found for any posetal diagram shape, which is terminal in a category of anticocones; furthermore, if any anticolimit exists, then the canonical anticocone will also be an anticolimit.
Unpacking this for antipushouts, we obtain the following corollary:
if a pair of morphisms have a pullback, then they have an antipushout iff the pullback is an antipushout:
\[
\begin{tikzcd}[column sep=small]
& C \ar[dd, phantom, "\llcorner" rotate=45, very near start] & \\
A \urar["f"] && B \ular["g"'] \\
& A \times_C B \ular["\pi_A"] \urar["\pi_B"'] \ar[uu, phantom, "\urcorner" rotate=45, very near start] &
\end{tikzcd}
\]
We also provide general techniques for constructing anticolimits from existing anticolimits, by changing either the shape of the diagrams, or the category in which they live.
We use these techniques to give an explicit construction procedure for anticolimits of sets, and then extend this to construct anticolimits of posets, preorders, and finite ordinals.

\paragraph{Application}
The motivating application for this work lies in the proof assistant \textsc{homotopy.io} for $n$-dimensional string diagrams, which is based on the theory of associative $n$-categories.
The main mechanism for proof construction is \emph{contraction}, presented at LICS 2019~\cite{RV19}, which uses a colimit operation to perform a homotopy contraction of part of the string diagram.
This is convenient when simplifying proof objects, but is far less elegant when we need to create nontrivial structure.

We illustrate this problem in \Cref{fig:naturality-move-v1}, showing the traditional proof assistant interaction steps required to construct the 4-cell representing the naturality move for the braiding.
A total of five user interaction steps are required, most of which are non-obvious.
With anticontraction, the procedure is far more straightforward; the new workflow is given in \Cref{fig:naturality-move-v2}, which shows that  only two  moves are now necessary, both of which are intuitively obvious.

For a more intricate example, we consider the 4-cell representing the Reidemeister III\ move from knot theory; in terms of the terminology of a braided monoidal category, this can be considered ``naturality for the braiding''.
Using anticontraction, we can simplify this workflow from fifteen steps (see \Cref{fig:reidemeister-move-v1}) down to seven steps (see \Cref{fig:reidemeister-move-v2}).

We develop the necessary theory for this in \Cref{sec:zigzags,sec:anticontraction}, including the relevant background on the theory of zigzag categories, which gives the inductive basis for the string diagram formalism.
We develop a series of results which show factorisation structures can be lifted through the zigzag construction, enabling anticontractions to the computed recursively.

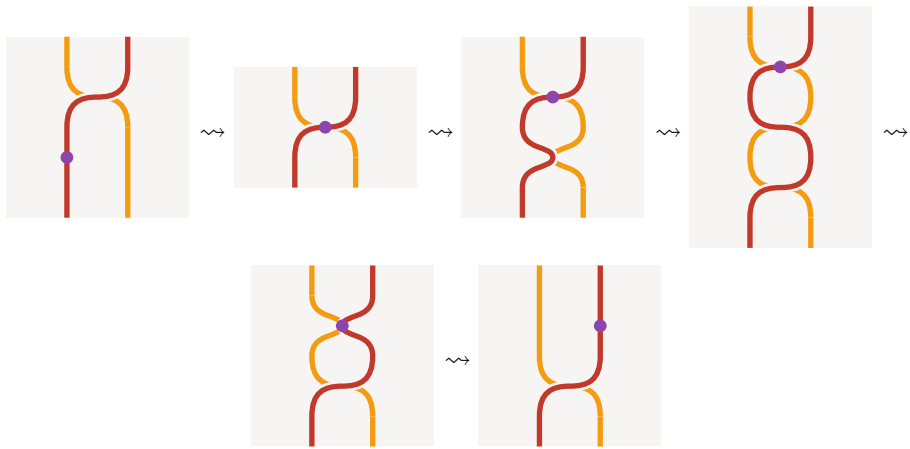
\begin{figure}
\centering
\(
\begin{aligned}\scalebox{0.4}{\input{figures/naturality_v1/regular0.tikz}}\end{aligned}
\leadsto
\begin{aligned}\scalebox{0.4}{\input{figures/naturality_v1/regular1.tikz}}\end{aligned}
\leadsto
\begin{aligned}\scalebox{0.4}{\input{figures/naturality_v1/regular2.tikz}}\end{aligned}
\leadsto
\begin{aligned}\scalebox{0.4}{\input{figures/naturality_v1/regular3.tikz}}\end{aligned}
\leadsto
\begin{aligned}\scalebox{0.4}{\input{figures/naturality_v1/regular4.tikz}}\end{aligned}
\leadsto
\begin{aligned}\scalebox{0.4}{\input{figures/naturality_v1/regular5.tikz}}\end{aligned}
\)
\caption{Naturality of the braiding without anticontraction.}
\label{fig:naturality-move-v1}
\end{figure}

\begin{figure}
\centering
\(
\begin{aligned}\scalebox{0.4}{\input{figures/naturality_v2/regular0.tikz}}\end{aligned}
\leadsto
\begin{aligned}\scalebox{0.4}{\input{figures/naturality_v2/regular1.tikz}}\end{aligned}
\leadsto
\begin{aligned}\scalebox{0.4}{\input{figures/naturality_v2/regular2.tikz}}\end{aligned}
\)
\caption{Naturality of the braiding with anticontraction.}
\label{fig:naturality-move-v2}
\end{figure}
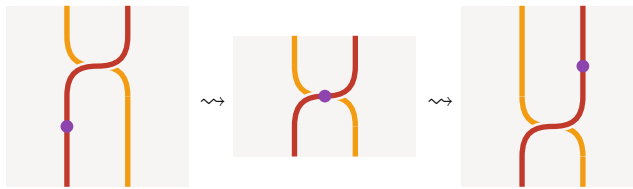

\begin{figure}
\centering
\(
\begin{aligned}\scalebox{0.3}{\input{figures/reidemeister_v1/regular0.tikz}}\end{aligned}
\leadsto
\begin{aligned}\scalebox{0.3}{\input{figures/reidemeister_v1/regular1.tikz}}\end{aligned}
\leadsto
\begin{aligned}\scalebox{0.3}{\input{figures/reidemeister_v1/regular2.tikz}}\end{aligned}
\leadsto
\begin{aligned}\scalebox{0.3}{\input{figures/reidemeister_v1/regular3.tikz}}\end{aligned}
\leadsto
\begin{aligned}\scalebox{0.3}{\input{figures/reidemeister_v1/regular4.tikz}}\end{aligned}
\leadsto
\begin{aligned}\scalebox{0.3}{\input{figures/reidemeister_v1/regular5.tikz}}\end{aligned}
\leadsto
\begin{aligned}\scalebox{0.3}{\input{figures/reidemeister_v1/regular6.tikz}}\end{aligned}
\leadsto
\begin{aligned}\scalebox{0.3}{\input{figures/reidemeister_v1/regular7.tikz}}\end{aligned}
\leadsto
\begin{aligned}\scalebox{0.3}{\input{figures/reidemeister_v1/regular8.tikz}}\end{aligned}
\leadsto
\begin{aligned}\scalebox{0.3}{\input{figures/reidemeister_v1/regular9.tikz}}\end{aligned}
\leadsto
\begin{aligned}\scalebox{0.3}{\input{figures/reidemeister_v1/regular10.tikz}}\end{aligned}
\leadsto
\begin{aligned}\scalebox{0.3}{\input{figures/reidemeister_v1/regular11.tikz}}\end{aligned}
\leadsto
\begin{aligned}\scalebox{0.3}{\input{figures/reidemeister_v1/regular12.tikz}}\end{aligned}
\leadsto
\begin{aligned}\scalebox{0.3}{\input{figures/reidemeister_v1/regular13.tikz}}\end{aligned}
\leadsto
\begin{aligned}\scalebox{0.3}{\input{figures/reidemeister_v1/regular14.tikz}}\end{aligned}
\leadsto
\begin{aligned}\scalebox{0.3}{\input{figures/reidemeister_v1/regular15.tikz}}\end{aligned}
\)
\caption{Reidemeister III  without anticontraction.}
\label{fig:reidemeister-move-v1}
\end{figure}
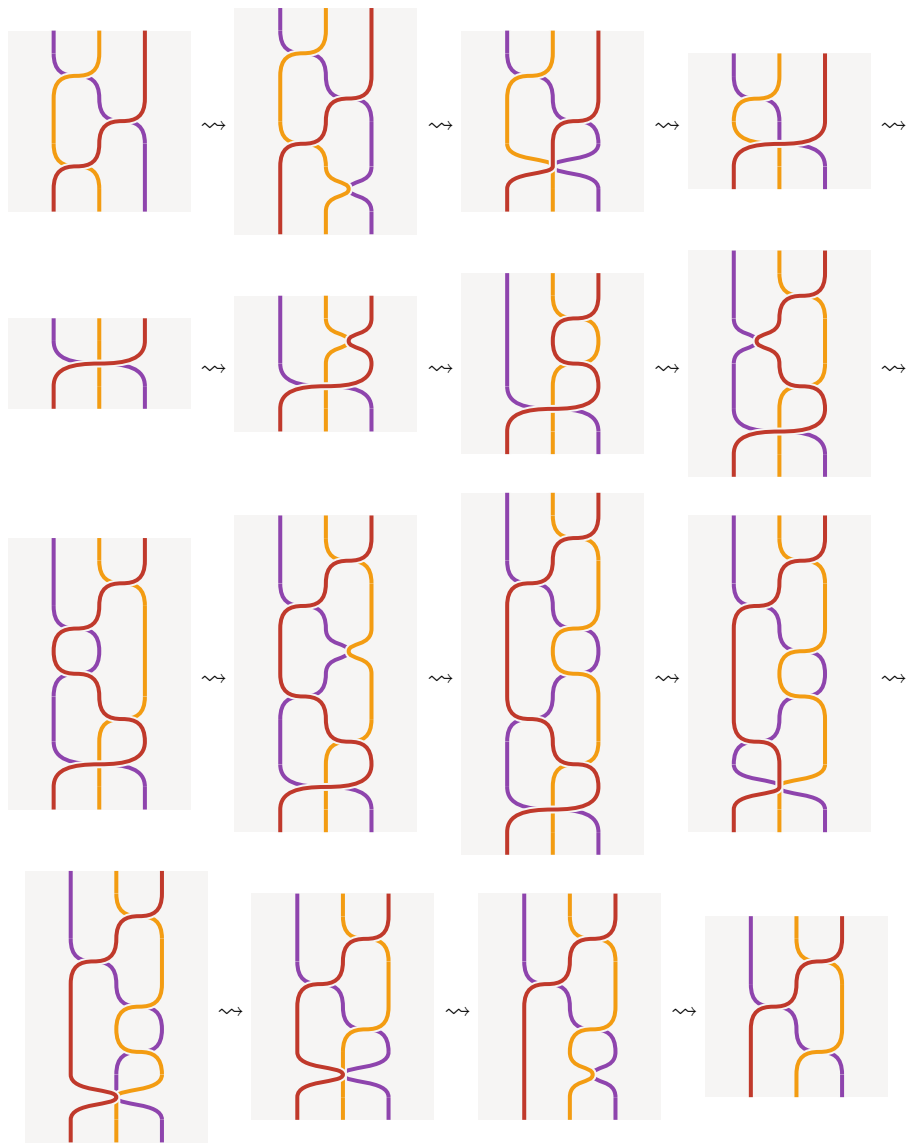

\begin{figure}
\centering
\(
\begin{aligned}\scalebox{0.3}{\input{figures/reidemeister_v2/regular0.tikz}}\end{aligned}
\leadsto
\begin{aligned}\scalebox{0.3}{\input{figures/reidemeister_v2/regular1.tikz}}\end{aligned}
\leadsto
\begin{aligned}\scalebox{0.3}{\input{figures/reidemeister_v2/regular2.tikz}}\end{aligned}
\leadsto
\begin{aligned}\scalebox{0.3}{\input{figures/reidemeister_v2/regular3.tikz}}\end{aligned}
\leadsto
\begin{aligned}\scalebox{0.3}{\input{figures/reidemeister_v2/regular4.tikz}}\end{aligned}
\leadsto
\begin{aligned}\scalebox{0.3}{\input{figures/reidemeister_v2/regular5.tikz}}\end{aligned}
\leadsto
\begin{aligned}\scalebox{0.3}{\input{figures/reidemeister_v2/regular6.tikz}}\end{aligned}
\leadsto
\begin{aligned}\scalebox{0.3}{\input{figures/reidemeister_v2/regular7.tikz}}\end{aligned}
\)
\caption{Reidemeister III  with anticontraction.}
\label{fig:reidemeister-move-v2}
\end{figure}
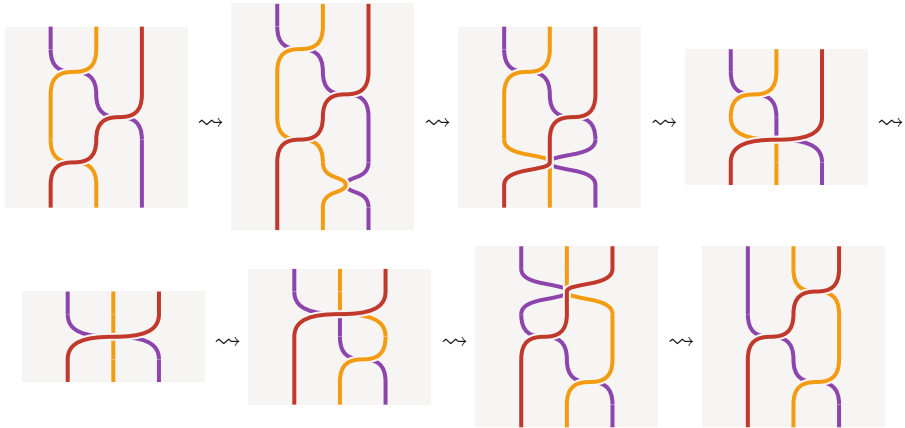

The authors have implemented anticontraction, and the feature is now available within a pre-release version of the proof assistant hosted at
\begin{center}
\url{https://beta.homotopy.io}
\end{center}
To use anticontraction, generators need to be marked as `oriented' in the signature, via the settings dialog for each generator; this disables certain geometrical  features  not currently compatible with anticontraction. To generate the moves of \Cref{fig:naturality-move-v2}, create the first image; then drag the vertex \emph{up} to create the second image; then drag the vertex \emph{right} to create the third image.

\subsection{Related work}

We are not aware of any work on anticolimits in the literature.
The closest relevant notion may be \emph{pushout complements} which play an important role in double pushout graph rewriting.
These resemble our notion of antipushouts, in that they both take as input one side of a pushout square and produce the other side:
\begin{itemize}
\item pushout complements take a composable pair of morphisms and produce the other composable pair;
\item antipushouts take two morphisms with a common codomain and produce two morphisms with a common domain.
\end{itemize}
We illustrate this as follows, where the solid arrows are given as input to the construction, and the dashed arrows are produced:
\[
\begin{aligned}
\begin{tikzcd}
\bullet \ar[dr, phantom, "\ulcorner", very near end] \dar[dashed] \rar & \bullet \dar \\
\bullet \rar[dashed] & \bullet
\end{tikzcd}
\end{aligned}
\!\!
\begin{array}{l}
\text{pushout} \\
\text{complement}
\end{array}
\hspace{.5cm}
\begin{aligned}
\begin{tikzcd}
\bullet \ar[dr, phantom, "\ulcorner", very near end] \dar[dashed] \rar[dashed] & \bullet \dar \\
\bullet \rar & \bullet
\end{tikzcd}
\end{aligned}
\!\!
\begin{array}{l}
\text{antipushout}
\end{array}
\]
The work on anticontraction builds on the existing body of work on associative $n$-categories and the proof assistant \textsc{homotopy.io}.
The theory of associative $n$\-categories was originally developed by Dorn, Douglas, and Vicary, and described in Dorn's PhD thesis~\cite{Dor18}.
We follow the development presented by Reutter and Vicary at LICS 2019~\cite{RV19}.
Further aspects of the proof assistant have been described by Corbyn, Heidemann, Hu, Sarti, Tataru, and Vicary~\cite{HRV22,Hei23,SV23,TV23}.

\subsection{Notation}

We adopt the following notation throughout the paper:
\begin{itemize}
\item $\ord n$ is the totally ordered set $\set{0, \dots, n - 1}$ for $n \in \mathbb{N}$.
\item $\Ord$ is the category of finite ordinals and monotone maps.
\item $\Int \subset \Ord$ is the category of finite intervals (i.e.\ finite non-empty ordinals) and monotone maps that preserve top and bottom.
\item $\Set$ is the category of sets and functions.
\item $\Pos$ is the category of posets and monotone maps.
\item $\Pre$ is the category of preorders and monotone maps.
\item $\Cat$ is the category of (small) categories and functors.
\item Given a poset $\pos J$ and an object $i \in \pos J$, define the following:
\[
\upper i \coloneqq \set{j \in \pos J \mid i \leq j} \qquad \Upper i \coloneqq \max(\upper i)
\]
\item $[\cat C, \cat D]$ is the functor category between categories $\cat C$ and $\cat D$.
\end{itemize}

\subsection{Acknowledgements}

We are grateful to Thibaut Benjamin, Lukas Heidemann, Nick Hu, Ioannis Markakis, Wilf Offord, and Chiara Sarti for useful discussions.

%% file: figures/naturality_v1/regular0.tikz
\begin{tikzpicture}
\definecolor{generator-1-2-0-pos}{RGB}{192, 57, 43}
\definecolor{generator-2-2-0-pos}{RGB}{243, 156, 18}
\definecolor{generator-3-3-0-pos}{RGB}{142, 68, 173}
\definecolor{generator-0-0-1-pos}{RGB}{246, 245, 244}

\newcommand{\wire}[2]{
  \ifdefined\recolor\draw[color=\recolor, line width=10pt]\else\draw[color=#1, line width=5pt]\fi #2;
}
\newcommand{\clipped}[3]{
\begin{scope}
  \newcommand{\recolor}{#1}
  \clip#3;
  #2
\end{scope}
}

\begin{scope}[transparency group]
% Background surfaces
\fill[generator-0-0-1-pos] (0,0) -- (6,0) -- (6,6) -- (0,6) -- (0,0);
\newcommand{\layer}[1]{
  \clipped{generator-0-0-1-pos}{#1}{(0,0) -- (6,0) -- (6,6) -- (0,6) -- (0,0)}
  #1
}

% Wire layers
\wire{generator-2-2-0-pos}{(4,3) .. controls (4,3.8) and (3.6,4) .. (3,4) .. controls (2.4,4) and (2,4.2) .. (2,5)};
\layer{
\wire{generator-1-2-0-pos}{(2,0) -- (2,3) .. controls (2,3.8) and (2.4,4) .. (3,4) .. controls (3.6,4) and (4,4.2) .. (4,5) -- (4,6)};
\wire{generator-2-2-0-pos}{(4,0) -- (4,3)(2,5) -- (2,6)};
}
\end{scope}
\fill[generator-3-3-0-pos] (2,2) circle (0.21);
\end{tikzpicture}

%% file: figures/naturality_v1/regular1.tikz
\begin{tikzpicture}
\definecolor{generator-1-2-0-pos}{RGB}{192, 57, 43}
\definecolor{generator-2-2-0-pos}{RGB}{243, 156, 18}
\definecolor{generator-3-3-0-pos}{RGB}{142, 68, 173}
\definecolor{generator-0-0-1-pos}{RGB}{246, 245, 244}

\newcommand{\wire}[2]{
  \ifdefined\recolor\draw[color=\recolor, line width=10pt]\else\draw[color=#1, line width=5pt]\fi #2;
}
\newcommand{\clipped}[3]{
\begin{scope}
  \newcommand{\recolor}{#1}
  \clip#3;
  #2
\end{scope}
}

\begin{scope}[transparency group]
% Background surfaces
\fill[generator-0-0-1-pos] (0,0) -- (6,0) -- (6,4) -- (0,4) -- (0,0);
\newcommand{\layer}[1]{
  \clipped{generator-0-0-1-pos}{#1}{(0,0) -- (6,0) -- (6,4) -- (0,4) -- (0,0)}
  #1
}

% Wire layers
\wire{generator-2-2-0-pos}{(4,1) .. controls (4,1.8) and (3.6,2) .. (3,2) .. controls (2.4,2) and (2,2.2) .. (2,3)};
\layer{
\wire{generator-1-2-0-pos}{(2,0) -- (2,1) .. controls (2,1.8) and (2.4,2) .. (3,2) .. controls (3.6,2) and (4,2.2) .. (4,3) -- (4,4)};
\wire{generator-2-2-0-pos}{(4,0) -- (4,1)(2,3) -- (2,4)};
}
\end{scope}
\fill[generator-3-3-0-pos] (3,2) circle (0.21);
\end{tikzpicture}

%% file: figures/naturality_v1/regular2.tikz
\begin{tikzpicture}
\definecolor{generator-1-2-0-pos}{RGB}{192, 57, 43}
\definecolor{generator-2-2-0-pos}{RGB}{243, 156, 18}
\definecolor{generator-3-3-0-pos}{RGB}{142, 68, 173}
\definecolor{generator-0-0-1-pos}{RGB}{246, 245, 244}

\newcommand{\wire}[2]{
  \ifdefined\recolor\draw[color=\recolor, line width=10pt]\else\draw[color=#1, line width=5pt]\fi #2;
}
\newcommand{\clipped}[3]{
\begin{scope}
  \newcommand{\recolor}{#1}
  \clip#3;
  #2
\end{scope}
}

\begin{scope}[transparency group]
% Background surfaces
\fill[generator-0-0-1-pos] (0,0) -- (6,0) -- (6,6) -- (0,6) -- (0,0);
\newcommand{\layer}[1]{
  \clipped{generator-0-0-1-pos}{#1}{(0,0) -- (6,0) -- (6,6) -- (0,6) -- (0,0)}
  #1
}

% Wire layers
\wire{generator-2-2-0-pos}{(4,1) .. controls (4,1.8) and (3,1.6) .. (3,2) .. controls (3,2.4) and (4,2.2) .. (4,3) .. controls (4,3.8) and (3.6,4) .. (3,4) .. controls (2.4,4) and (2,4.2) .. (2,5)};
\layer{
\wire{generator-1-2-0-pos}{(2,0) -- (2,1) .. controls (2,1.8) and (3,1.6) .. (3,2) .. controls (3,2.4) and (2,2.2) .. (2,3) .. controls (2,3.8) and (2.4,4) .. (3,4) .. controls (3.6,4) and (4,4.2) .. (4,5) -- (4,6)};
\wire{generator-2-2-0-pos}{(4,0) -- (4,1)(2,5) -- (2,6)};
}
\end{scope}
\fill[generator-3-3-0-pos] (3,4) circle (0.21);
\end{tikzpicture}

%% file: figures/naturality_v1/regular3.tikz
\begin{tikzpicture}
\definecolor{generator-1-2-0-pos}{RGB}{192, 57, 43}
\definecolor{generator-2-2-0-pos}{RGB}{243, 156, 18}
\definecolor{generator-3-3-0-pos}{RGB}{142, 68, 173}
\definecolor{generator-0-0-1-pos}{RGB}{246, 245, 244}

\newcommand{\wire}[2]{
  \ifdefined\recolor\draw[color=\recolor, line width=10pt]\else\draw[color=#1, line width=5pt]\fi #2;
}
\newcommand{\clipped}[3]{
\begin{scope}
  \newcommand{\recolor}{#1}
  \clip#3;
  #2
\end{scope}
}

\begin{scope}[transparency group]
% Background surfaces
\fill[generator-0-0-1-pos] (0,0) -- (6,0) -- (6,8) -- (0,8) -- (0,0);
\newcommand{\layer}[1]{
  \clipped{generator-0-0-1-pos}{#1}{(0,0) -- (6,0) -- (6,8) -- (0,8) -- (0,0)}
  #1
}

% Wire layers
\wire{generator-2-2-0-pos}{(4,1) .. controls (4,1.8) and (3.6,2) .. (3,2) .. controls (2.4,2) and (2,2.2) .. (2,3) .. controls (2,3.8) and (2.4,4) .. (3,4) .. controls (3.6,4) and (4,4.2) .. (4,5) .. controls (4,5.8) and (3.6,6) .. (3,6) .. controls (2.4,6) and (2,6.2) .. (2,7)};
\layer{
\wire{generator-1-2-0-pos}{(2,0) -- (2,1) .. controls (2,1.8) and (2.4,2) .. (3,2) .. controls (3.6,2) and (4,2.2) .. (4,3) .. controls (4,3.8) and (3.6,4) .. (3,4) .. controls (2.4,4) and (2,4.2) .. (2,5) .. controls (2,5.8) and (2.4,6) .. (3,6) .. controls (3.6,6) and (4,6.2) .. (4,7) -- (4,8)};
\wire{generator-2-2-0-pos}{(4,0) -- (4,1)(2,7) -- (2,8)};
}
\end{scope}
\fill[generator-3-3-0-pos] (3,6) circle (0.21);
\end{tikzpicture}

%% file: figures/naturality_v1/regular4.tikz
\begin{tikzpicture}
\definecolor{generator-1-2-0-pos}{RGB}{192, 57, 43}
\definecolor{generator-2-2-0-pos}{RGB}{243, 156, 18}
\definecolor{generator-3-3-0-pos}{RGB}{142, 68, 173}
\definecolor{generator-0-0-1-pos}{RGB}{246, 245, 244}

\newcommand{\wire}[2]{
  \ifdefined\recolor\draw[color=\recolor, line width=10pt]\else\draw[color=#1, line width=5pt]\fi #2;
}
\newcommand{\clipped}[3]{
\begin{scope}
  \newcommand{\recolor}{#1}
  \clip#3;
  #2
\end{scope}
}

\begin{scope}[transparency group]
% Background surfaces
\fill[generator-0-0-1-pos] (0,0) -- (6,0) -- (6,6) -- (0,6) -- (0,0);
\newcommand{\layer}[1]{
  \clipped{generator-0-0-1-pos}{#1}{(0,0) -- (6,0) -- (6,6) -- (0,6) -- (0,0)}
  #1
}

% Wire layers
\wire{generator-2-2-0-pos}{(4,1) .. controls (4,1.8) and (3.6,2) .. (3,2) .. controls (2.4,2) and (2,2.2) .. (2,3) .. controls (2,3.8) and (3,3.6) .. (3,4) .. controls (3,4.4) and (2,4.2) .. (2,5)};
\layer{
\wire{generator-1-2-0-pos}{(2,0) -- (2,1) .. controls (2,1.8) and (2.4,2) .. (3,2) .. controls (3.6,2) and (4,2.2) .. (4,3) .. controls (4,3.8) and (3,3.6) .. (3,4) .. controls (3,4.4) and (4,4.2) .. (4,5) -- (4,6)};
\wire{generator-2-2-0-pos}{(4,0) -- (4,1)(2,5) -- (2,6)};
}
\end{scope}
\fill[generator-3-3-0-pos] (3,4) circle (0.21);
\end{tikzpicture}

%% file: figures/naturality_v1/regular5.tikz
\begin{tikzpicture}
\definecolor{generator-1-2-0-pos}{RGB}{192, 57, 43}
\definecolor{generator-2-2-0-pos}{RGB}{243, 156, 18}
\definecolor{generator-3-3-0-pos}{RGB}{142, 68, 173}
\definecolor{generator-0-0-1-pos}{RGB}{246, 245, 244}

\newcommand{\wire}[2]{
  \ifdefined\recolor\draw[color=\recolor, line width=10pt]\else\draw[color=#1, line width=5pt]\fi #2;
}
\newcommand{\clipped}[3]{
\begin{scope}
  \newcommand{\recolor}{#1}
  \clip#3;
  #2
\end{scope}
}

\begin{scope}[transparency group]
% Background surfaces
\fill[generator-0-0-1-pos] (0,0) -- (6,0) -- (6,6) -- (0,6) -- (0,0);
\newcommand{\layer}[1]{
  \clipped{generator-0-0-1-pos}{#1}{(0,0) -- (6,0) -- (6,6) -- (0,6) -- (0,0)}
  #1
}

% Wire layers
\wire{generator-2-2-0-pos}{(4,1) .. controls (4,1.8) and (3.6,2) .. (3,2) .. controls (2.4,2) and (2,2.2) .. (2,3)};
\layer{
\wire{generator-1-2-0-pos}{(2,0) -- (2,1) .. controls (2,1.8) and (2.4,2) .. (3,2) .. controls (3.6,2) and (4,2.2) .. (4,3) -- (4,6)};
\wire{generator-2-2-0-pos}{(4,0) -- (4,1)(2,3) -- (2,6)};
}
\end{scope}
\fill[generator-3-3-0-pos] (4,4) circle (0.21);
\end{tikzpicture}

%% file: figures/naturality_v2/regular0.tikz
\begin{tikzpicture}
\definecolor{generator-1-2-0-pos}{RGB}{192, 57, 43}
\definecolor{generator-2-2-0-pos}{RGB}{243, 156, 18}
\definecolor{generator-3-3-0-pos}{RGB}{142, 68, 173}
\definecolor{generator-0-0-1-pos}{RGB}{246, 245, 244}

\newcommand{\wire}[2]{
  \ifdefined\recolor\draw[color=\recolor, line width=10pt]\else\draw[color=#1, line width=5pt]\fi #2;
}
\newcommand{\clipped}[3]{
\begin{scope}
  \newcommand{\recolor}{#1}
  \clip#3;
  #2
\end{scope}
}

\begin{scope}[transparency group]
% Background surfaces
\fill[generator-0-0-1-pos] (0,0) -- (6,0) -- (6,6) -- (0,6) -- (0,0);
\newcommand{\layer}[1]{
  \clipped{generator-0-0-1-pos}{#1}{(0,0) -- (6,0) -- (6,6) -- (0,6) -- (0,0)}
  #1
}

% Wire layers
\wire{generator-2-2-0-pos}{(4,3) .. controls (4,3.8) and (3.6,4) .. (3,4) .. controls (2.4,4) and (2,4.2) .. (2,5)};
\layer{
\wire{generator-1-2-0-pos}{(2,0) -- (2,3) .. controls (2,3.8) and (2.4,4) .. (3,4) .. controls (3.6,4) and (4,4.2) .. (4,5) -- (4,6)};
\wire{generator-2-2-0-pos}{(4,0) -- (4,3)(2,5) -- (2,6)};
}
\end{scope}
\fill[generator-3-3-0-pos] (2,2) circle (0.21);
\end{tikzpicture}

%% file: figures/reidemeister_v1/regular0.tikz
\begin{tikzpicture}
\definecolor{generator-3-2-0-pos}{RGB}{142, 68, 173}
\definecolor{generator-1-2-0-pos}{RGB}{192, 57, 43}
\definecolor{generator-2-2-0-pos}{RGB}{243, 156, 18}
\definecolor{generator-0-0-1-pos}{RGB}{246, 245, 244}

\newcommand{\wire}[2]{
  \ifdefined\recolor\draw[color=\recolor, line width=10pt]\else\draw[color=#1, line width=5pt]\fi #2;
}
\newcommand{\clipped}[3]{
\begin{scope}
  \newcommand{\recolor}{#1}
  \clip#3;
  #2
\end{scope}
}

\begin{scope}[transparency group]
% Background surfaces
\fill[generator-0-0-1-pos] (0,0) -- (8,0) -- (8,8) -- (0,8) -- (0,0);
\newcommand{\layer}[1]{
  \clipped{generator-0-0-1-pos}{#1}{(0,0) -- (8,0) -- (8,8) -- (0,8) -- (0,0)}
  #1
}

% Wire layers
\wire{generator-2-2-0-pos}{(4,1) .. controls (4,1.8) and (3.6,2) .. (3,2) .. controls (2.4,2) and (2,2.2) .. (2,3)};
\wire{generator-3-2-0-pos}{(6,3) .. controls (6,3.8) and (5.6,4) .. (5,4) .. controls (4.4,4) and (4,4.2) .. (4,5) .. controls (4,5.8) and (3.6,6) .. (3,6) .. controls (2.4,6) and (2,6.2) .. (2,7)};
\layer{
\wire{generator-3-2-0-pos}{(6,0) -- (6,3)(2,7) -- (2,8)};
\wire{generator-1-2-0-pos}{(2,0) -- (2,1) .. controls (2,1.8) and (2.4,2) .. (3,2) .. controls (3.6,2) and (4,2.2) .. (4,3) .. controls (4,3.8) and (4.4,4) .. (5,4) .. controls (5.6,4) and (6,4.2) .. (6,5) -- (6,8)};
\wire{generator-2-2-0-pos}{(4,0) -- (4,1)(2,3) -- (2,5) .. controls (2,5.8) and (2.4,6) .. (3,6) .. controls (3.6,6) and (4,6.2) .. (4,7) -- (4,8)};
}
\end{scope}
\end{tikzpicture}

%% file: figures/reidemeister_v1/regular1.tikz
\begin{tikzpicture}
\definecolor{generator-3-2-0-pos}{RGB}{142, 68, 173}
\definecolor{generator-1-2-0-pos}{RGB}{192, 57, 43}
\definecolor{generator-2-2-0-pos}{RGB}{243, 156, 18}
\definecolor{generator-0-0-1-pos}{RGB}{246, 245, 244}

\newcommand{\wire}[2]{
  \ifdefined\recolor\draw[color=\recolor, line width=10pt]\else\draw[color=#1, line width=5pt]\fi #2;
}
\newcommand{\clipped}[3]{
\begin{scope}
  \newcommand{\recolor}{#1}
  \clip#3;
  #2
\end{scope}
}

\begin{scope}[transparency group]
% Background surfaces
\fill[generator-0-0-1-pos] (0,0) -- (8,0) -- (8,10) -- (0,10) -- (0,0);
\newcommand{\layer}[1]{
  \clipped{generator-0-0-1-pos}{#1}{(0,0) -- (8,0) -- (8,10) -- (0,10) -- (0,0)}
  #1
}

% Wire layers
\wire{generator-2-2-0-pos}{(4,3) .. controls (4,3.8) and (3.6,4) .. (3,4) .. controls (2.4,4) and (2,4.2) .. (2,5)};
\wire{generator-3-2-0-pos}{(6,1) .. controls (6,1.8) and (5,1.6) .. (5,2) .. controls (5,2.4) and (6,2.2) .. (6,3)(6,5) .. controls (6,5.8) and (5.6,6) .. (5,6) .. controls (4.4,6) and (4,6.2) .. (4,7) .. controls (4,7.8) and (3.6,8) .. (3,8) .. controls (2.4,8) and (2,8.2) .. (2,9)};
\layer{
\wire{generator-3-2-0-pos}{(6,0) -- (6,1)(6,3) -- (6,5)(2,9) -- (2,10)};
\wire{generator-1-2-0-pos}{(2,0) -- (2,3) .. controls (2,3.8) and (2.4,4) .. (3,4) .. controls (3.6,4) and (4,4.2) .. (4,5) .. controls (4,5.8) and (4.4,6) .. (5,6) .. controls (5.6,6) and (6,6.2) .. (6,7) -- (6,10)};
\wire{generator-2-2-0-pos}{(4,0) -- (4,1) .. controls (4,1.8) and (5,1.6) .. (5,2) .. controls (5,2.4) and (4,2.2) .. (4,3)(2,5) -- (2,7) .. controls (2,7.8) and (2.4,8) .. (3,8) .. controls (3.6,8) and (4,8.2) .. (4,9) -- (4,10)};
}
\end{scope}
\end{tikzpicture}

%% file: figures/reidemeister_v1/regular2.tikz
\begin{tikzpicture}
\definecolor{generator-3-2-0-pos}{RGB}{142, 68, 173}
\definecolor{generator-1-2-0-pos}{RGB}{192, 57, 43}
\definecolor{generator-2-2-0-pos}{RGB}{243, 156, 18}
\definecolor{generator-0-0-1-pos}{RGB}{246, 245, 244}

\newcommand{\wire}[2]{
  \ifdefined\recolor\draw[color=\recolor, line width=10pt]\else\draw[color=#1, line width=5pt]\fi #2;
}
\newcommand{\clipped}[3]{
\begin{scope}
  \newcommand{\recolor}{#1}
  \clip#3;
  #2
\end{scope}
}

\begin{scope}[transparency group]
% Background surfaces
\fill[generator-0-0-1-pos] (0,0) -- (8,0) -- (8,8) -- (0,8) -- (0,0);
\newcommand{\layer}[1]{
  \clipped{generator-0-0-1-pos}{#1}{(0,0) -- (8,0) -- (8,8) -- (0,8) -- (0,0)}
  #1
}

% Wire layers
\wire{generator-3-2-0-pos}{(6,1) .. controls (6,1.8) and (4,1.6) .. (4,2) .. controls (4,2.4) and (6,2.2) .. (6,3)};
\layer{
\wire{generator-2-2-0-pos}{(4,1) -- (4,2) .. controls (4,2.4) and (2,2.2) .. (2,3)};
\wire{generator-3-2-0-pos}{(6,3) .. controls (6,3.8) and (5.6,4) .. (5,4) .. controls (4.4,4) and (4,4.2) .. (4,5) .. controls (4,5.8) and (3.6,6) .. (3,6) .. controls (2.4,6) and (2,6.2) .. (2,7)};
}
\layer{
\wire{generator-3-2-0-pos}{(6,0) -- (6,1)(2,7) -- (2,8)};
\wire{generator-1-2-0-pos}{(2,0) -- (2,1) .. controls (2,1.8) and (4,1.6) .. (4,2) -- (4,3) .. controls (4,3.8) and (4.4,4) .. (5,4) .. controls (5.6,4) and (6,4.2) .. (6,5) -- (6,8)};
\wire{generator-2-2-0-pos}{(4,0) -- (4,1)(2,3) -- (2,5) .. controls (2,5.8) and (2.4,6) .. (3,6) .. controls (3.6,6) and (4,6.2) .. (4,7) -- (4,8)};
}
\end{scope}
\end{tikzpicture}

%% file: figures/reidemeister_v1/regular3.tikz
\begin{tikzpicture}
\definecolor{generator-3-2-0-pos}{RGB}{142, 68, 173}
\definecolor{generator-1-2-0-pos}{RGB}{192, 57, 43}
\definecolor{generator-2-2-0-pos}{RGB}{243, 156, 18}
\definecolor{generator-0-0-1-pos}{RGB}{246, 245, 244}

\newcommand{\wire}[2]{
  \ifdefined\recolor\draw[color=\recolor, line width=10pt]\else\draw[color=#1, line width=5pt]\fi #2;
}
\newcommand{\clipped}[3]{
\begin{scope}
  \newcommand{\recolor}{#1}
  \clip#3;
  #2
\end{scope}
}

\begin{scope}[transparency group]
% Background surfaces
\fill[generator-0-0-1-pos] (0,0) -- (8,0) -- (8,6) -- (0,6) -- (0,0);
\newcommand{\layer}[1]{
  \clipped{generator-0-0-1-pos}{#1}{(0,0) -- (8,0) -- (8,6) -- (0,6) -- (0,0)}
  #1
}

% Wire layers
\wire{generator-3-2-0-pos}{(6,1) .. controls (6,1.8) and (5.2,2) .. (4,2) -- (4,3)};
\layer{
\wire{generator-2-2-0-pos}{(4,1) -- (4,2) .. controls (2.8,2) and (2,2.2) .. (2,3)};
\wire{generator-3-2-0-pos}{(4,3) .. controls (4,3.8) and (3.6,4) .. (3,4) .. controls (2.4,4) and (2,4.2) .. (2,5)};
}
\layer{
\wire{generator-3-2-0-pos}{(6,0) -- (6,1)(2,5) -- (2,6)};
\wire{generator-1-2-0-pos}{(2,0) -- (2,1) .. controls (2,1.8) and (2.8,2) .. (4,2) .. controls (5.2,2) and (6,2.2) .. (6,3) -- (6,6)};
\wire{generator-2-2-0-pos}{(4,0) -- (4,1)(2,3) .. controls (2,3.8) and (2.4,4) .. (3,4) .. controls (3.6,4) and (4,4.2) .. (4,5) -- (4,6)};
}
\end{scope}
\end{tikzpicture}

%% file: figures/reidemeister_v1/regular4.tikz
\begin{tikzpicture}
\definecolor{generator-3-2-0-pos}{RGB}{142, 68, 173}
\definecolor{generator-1-2-0-pos}{RGB}{192, 57, 43}
\definecolor{generator-2-2-0-pos}{RGB}{243, 156, 18}
\definecolor{generator-0-0-1-pos}{RGB}{246, 245, 244}

\newcommand{\wire}[2]{
  \ifdefined\recolor\draw[color=\recolor, line width=10pt]\else\draw[color=#1, line width=5pt]\fi #2;
}
\newcommand{\clipped}[3]{
\begin{scope}
  \newcommand{\recolor}{#1}
  \clip#3;
  #2
\end{scope}
}

\begin{scope}[transparency group]
% Background surfaces
\fill[generator-0-0-1-pos] (0,0) -- (8,0) -- (8,4) -- (0,4) -- (0,0);
\newcommand{\layer}[1]{
  \clipped{generator-0-0-1-pos}{#1}{(0,0) -- (8,0) -- (8,4) -- (0,4) -- (0,0)}
  #1
}

% Wire layers
\wire{generator-3-2-0-pos}{(6,1) .. controls (6,1.8) and (5.2,2) .. (4,2) .. controls (2.8,2) and (2,2.2) .. (2,3)};
\layer{
\wire{generator-2-2-0-pos}{(4,1) -- (4,3)};
}
\layer{
\wire{generator-3-2-0-pos}{(6,0) -- (6,1)(2,3) -- (2,4)};
\wire{generator-1-2-0-pos}{(2,0) -- (2,1) .. controls (2,1.8) and (2.8,2) .. (4,2) .. controls (5.2,2) and (6,2.2) .. (6,3) -- (6,4)};
\wire{generator-2-2-0-pos}{(4,0) -- (4,1)(4,3) -- (4,4)};
}
\end{scope}
\end{tikzpicture}

%% file: figures/reidemeister_v1/regular5.tikz
\begin{tikzpicture}
\definecolor{generator-3-2-0-pos}{RGB}{142, 68, 173}
\definecolor{generator-1-2-0-pos}{RGB}{192, 57, 43}
\definecolor{generator-2-2-0-pos}{RGB}{243, 156, 18}
\definecolor{generator-0-0-1-pos}{RGB}{246, 245, 244}

\newcommand{\wire}[2]{
  \ifdefined\recolor\draw[color=\recolor, line width=10pt]\else\draw[color=#1, line width=5pt]\fi #2;
}
\newcommand{\clipped}[3]{
\begin{scope}
  \newcommand{\recolor}{#1}
  \clip#3;
  #2
\end{scope}
}

\begin{scope}[transparency group]
% Background surfaces
\fill[generator-0-0-1-pos] (0,0) -- (8,0) -- (8,6) -- (0,6) -- (0,0);
\newcommand{\layer}[1]{
  \clipped{generator-0-0-1-pos}{#1}{(0,0) -- (8,0) -- (8,6) -- (0,6) -- (0,0)}
  #1
}

% Wire layers
\wire{generator-3-2-0-pos}{(6,1) .. controls (6,1.8) and (5.2,2) .. (4,2) .. controls (2.8,2) and (2,2.2) .. (2,3)};
\layer{
\wire{generator-2-2-0-pos}{(4,1) -- (4,3) .. controls (4,3.8) and (5,3.6) .. (5,4) .. controls (5,4.4) and (4,4.2) .. (4,5)};
}
\layer{
\wire{generator-3-2-0-pos}{(6,0) -- (6,1)(2,3) -- (2,6)};
\wire{generator-1-2-0-pos}{(2,0) -- (2,1) .. controls (2,1.8) and (2.8,2) .. (4,2) .. controls (5.2,2) and (6,2.2) .. (6,3) .. controls (6,3.8) and (5,3.6) .. (5,4) .. controls (5,4.4) and (6,4.2) .. (6,5) -- (6,6)};
\wire{generator-2-2-0-pos}{(4,0) -- (4,1)(4,5) -- (4,6)};
}
\end{scope}
\end{tikzpicture}

%% file: figures/reidemeister_v1/regular6.tikz
\begin{tikzpicture}
\definecolor{generator-3-2-0-pos}{RGB}{142, 68, 173}
\definecolor{generator-1-2-0-pos}{RGB}{192, 57, 43}
\definecolor{generator-2-2-0-pos}{RGB}{243, 156, 18}
\definecolor{generator-0-0-1-pos}{RGB}{246, 245, 244}

\newcommand{\wire}[2]{
  \ifdefined\recolor\draw[color=\recolor, line width=10pt]\else\draw[color=#1, line width=5pt]\fi #2;
}
\newcommand{\clipped}[3]{
\begin{scope}
  \newcommand{\recolor}{#1}
  \clip#3;
  #2
\end{scope}
}

\begin{scope}[transparency group]
% Background surfaces
\fill[generator-0-0-1-pos] (0,0) -- (8,0) -- (8,8) -- (0,8) -- (0,0);
\newcommand{\layer}[1]{
  \clipped{generator-0-0-1-pos}{#1}{(0,0) -- (8,0) -- (8,8) -- (0,8) -- (0,0)}
  #1
}

% Wire layers
\wire{generator-3-2-0-pos}{(6,1) .. controls (6,1.8) and (5.2,2) .. (4,2) .. controls (2.8,2) and (2,2.2) .. (2,3)};
\layer{
\wire{generator-2-2-0-pos}{(4,1) -- (4,3) .. controls (4,3.8) and (4.4,4) .. (5,4) .. controls (5.6,4) and (6,4.2) .. (6,5) .. controls (6,5.8) and (5.6,6) .. (5,6) .. controls (4.4,6) and (4,6.2) .. (4,7)};
}
\layer{
\wire{generator-3-2-0-pos}{(6,0) -- (6,1)(2,3) -- (2,8)};
\wire{generator-1-2-0-pos}{(2,0) -- (2,1) .. controls (2,1.8) and (2.8,2) .. (4,2) .. controls (5.2,2) and (6,2.2) .. (6,3) .. controls (6,3.8) and (5.6,4) .. (5,4) .. controls (4.4,4) and (4,4.2) .. (4,5) .. controls (4,5.8) and (4.4,6) .. (5,6) .. controls (5.6,6) and (6,6.2) .. (6,7) -- (6,8)};
\wire{generator-2-2-0-pos}{(4,0) -- (4,1)(4,7) -- (4,8)};
}
\end{scope}
\end{tikzpicture}

%% file: figures/reidemeister_v1/regular7.tikz
\begin{tikzpicture}
\definecolor{generator-3-2-0-pos}{RGB}{142, 68, 173}
\definecolor{generator-1-2-0-pos}{RGB}{192, 57, 43}
\definecolor{generator-2-2-0-pos}{RGB}{243, 156, 18}
\definecolor{generator-0-0-1-pos}{RGB}{246, 245, 244}

\newcommand{\wire}[2]{
  \ifdefined\recolor\draw[color=\recolor, line width=10pt]\else\draw[color=#1, line width=5pt]\fi #2;
}
\newcommand{\clipped}[3]{
\begin{scope}
  \newcommand{\recolor}{#1}
  \clip#3;
  #2
\end{scope}
}

\begin{scope}[transparency group]
% Background surfaces
\fill[generator-0-0-1-pos] (0,0) -- (8,0) -- (8,10) -- (0,10) -- (0,0);
\newcommand{\layer}[1]{
  \clipped{generator-0-0-1-pos}{#1}{(0,0) -- (8,0) -- (8,10) -- (0,10) -- (0,0)}
  #1
}

% Wire layers
\wire{generator-3-2-0-pos}{(6,1) .. controls (6,1.8) and (5.2,2) .. (4,2) .. controls (2.8,2) and (2,2.2) .. (2,3)};
\layer{
\wire{generator-2-2-0-pos}{(4,1) -- (4,3) .. controls (4,3.8) and (4.4,4) .. (5,4) .. controls (5.6,4) and (6,4.2) .. (6,5)(6,7) .. controls (6,7.8) and (5.6,8) .. (5,8) .. controls (4.4,8) and (4,8.2) .. (4,9)};
\wire{generator-3-2-0-pos}{(2,5) .. controls (2,5.8) and (3,5.6) .. (3,6) .. controls (3,6.4) and (2,6.2) .. (2,7)};
}
\layer{
\wire{generator-3-2-0-pos}{(6,0) -- (6,1)(2,3) -- (2,5)(2,7) -- (2,10)};
\wire{generator-1-2-0-pos}{(2,0) -- (2,1) .. controls (2,1.8) and (2.8,2) .. (4,2) .. controls (5.2,2) and (6,2.2) .. (6,3) .. controls (6,3.8) and (5.6,4) .. (5,4) .. controls (4.4,4) and (4,4.2) .. (4,5) .. controls (4,5.8) and (3,5.6) .. (3,6) .. controls (3,6.4) and (4,6.2) .. (4,7) .. controls (4,7.8) and (4.4,8) .. (5,8) .. controls (5.6,8) and (6,8.2) .. (6,9) -- (6,10)};
\wire{generator-2-2-0-pos}{(4,0) -- (4,1)(6,5) -- (6,7)(4,9) -- (4,10)};
}
\end{scope}
\end{tikzpicture}

%% file: figures/reidemeister_v1/regular8.tikz
\begin{tikzpicture}
\definecolor{generator-3-2-0-pos}{RGB}{142, 68, 173}
\definecolor{generator-1-2-0-pos}{RGB}{192, 57, 43}
\definecolor{generator-2-2-0-pos}{RGB}{243, 156, 18}
\definecolor{generator-0-0-1-pos}{RGB}{246, 245, 244}

\newcommand{\wire}[2]{
  \ifdefined\recolor\draw[color=\recolor, line width=10pt]\else\draw[color=#1, line width=5pt]\fi #2;
}
\newcommand{\clipped}[3]{
\begin{scope}
  \newcommand{\recolor}{#1}
  \clip#3;
  #2
\end{scope}
}

\begin{scope}[transparency group]
% Background surfaces
\fill[generator-0-0-1-pos] (0,0) -- (8,0) -- (8,12) -- (0,12) -- (0,0);
\newcommand{\layer}[1]{
  \clipped{generator-0-0-1-pos}{#1}{(0,0) -- (8,0) -- (8,12) -- (0,12) -- (0,0)}
  #1
}

% Wire layers
\wire{generator-3-2-0-pos}{(6,1) .. controls (6,1.8) and (5.2,2) .. (4,2) .. controls (2.8,2) and (2,2.2) .. (2,3)};
\layer{
\wire{generator-2-2-0-pos}{(4,1) -- (4,3) .. controls (4,3.8) and (4.4,4) .. (5,4) .. controls (5.6,4) and (6,4.2) .. (6,5)(6,9) .. controls (6,9.8) and (5.6,10) .. (5,10) .. controls (4.4,10) and (4,10.2) .. (4,11)};
\wire{generator-3-2-0-pos}{(2,5) .. controls (2,5.8) and (2.4,6) .. (3,6) .. controls (3.6,6) and (4,6.2) .. (4,7) .. controls (4,7.8) and (3.6,8) .. (3,8) .. controls (2.4,8) and (2,8.2) .. (2,9)};
}
\layer{
\wire{generator-3-2-0-pos}{(6,0) -- (6,1)(2,3) -- (2,5)(2,9) -- (2,12)};
\wire{generator-1-2-0-pos}{(2,0) -- (2,1) .. controls (2,1.8) and (2.8,2) .. (4,2) .. controls (5.2,2) and (6,2.2) .. (6,3) .. controls (6,3.8) and (5.6,4) .. (5,4) .. controls (4.4,4) and (4,4.2) .. (4,5) .. controls (4,5.8) and (3.6,6) .. (3,6) .. controls (2.4,6) and (2,6.2) .. (2,7) .. controls (2,7.8) and (2.4,8) .. (3,8) .. controls (3.6,8) and (4,8.2) .. (4,9) .. controls (4,9.8) and (4.4,10) .. (5,10) .. controls (5.6,10) and (6,10.2) .. (6,11) -- (6,12)};
\wire{generator-2-2-0-pos}{(4,0) -- (4,1)(6,5) -- (6,9)(4,11) -- (4,12)};
}
\end{scope}
\end{tikzpicture}

%% file: figures/reidemeister_v1/regular9.tikz
\begin{tikzpicture}
\definecolor{generator-3-2-0-pos}{RGB}{142, 68, 173}
\definecolor{generator-1-2-0-pos}{RGB}{192, 57, 43}
\definecolor{generator-2-2-0-pos}{RGB}{243, 156, 18}
\definecolor{generator-0-0-1-pos}{RGB}{246, 245, 244}

\newcommand{\wire}[2]{
  \ifdefined\recolor\draw[color=\recolor, line width=10pt]\else\draw[color=#1, line width=5pt]\fi #2;
}
\newcommand{\clipped}[3]{
\begin{scope}
  \newcommand{\recolor}{#1}
  \clip#3;
  #2
\end{scope}
}

\begin{scope}[transparency group]
% Background surfaces
\fill[generator-0-0-1-pos] (0,0) -- (8,0) -- (8,14) -- (0,14) -- (0,0);
\newcommand{\layer}[1]{
  \clipped{generator-0-0-1-pos}{#1}{(0,0) -- (8,0) -- (8,14) -- (0,14) -- (0,0)}
  #1
}

% Wire layers
\wire{generator-3-2-0-pos}{(6,1) .. controls (6,1.8) and (5.2,2) .. (4,2) .. controls (2.8,2) and (2,2.2) .. (2,3)};
\layer{
\wire{generator-2-2-0-pos}{(4,1) -- (4,3) .. controls (4,3.8) and (4.4,4) .. (5,4) .. controls (5.6,4) and (6,4.2) .. (6,5)(6,11) .. controls (6,11.8) and (5.6,12) .. (5,12) .. controls (4.4,12) and (4,12.2) .. (4,13)};
\wire{generator-3-2-0-pos}{(2,5) .. controls (2,5.8) and (2.4,6) .. (3,6) .. controls (3.6,6) and (4,6.2) .. (4,7) .. controls (4,7.8) and (5,7.6) .. (5,8) .. controls (5,8.4) and (4,8.2) .. (4,9) .. controls (4,9.8) and (3.6,10) .. (3,10) .. controls (2.4,10) and (2,10.2) .. (2,11)};
}
\layer{
\wire{generator-3-2-0-pos}{(6,0) -- (6,1)(2,3) -- (2,5)(2,11) -- (2,14)};
\wire{generator-1-2-0-pos}{(2,0) -- (2,1) .. controls (2,1.8) and (2.8,2) .. (4,2) .. controls (5.2,2) and (6,2.2) .. (6,3) .. controls (6,3.8) and (5.6,4) .. (5,4) .. controls (4.4,4) and (4,4.2) .. (4,5) .. controls (4,5.8) and (3.6,6) .. (3,6) .. controls (2.4,6) and (2,6.2) .. (2,7) -- (2,9) .. controls (2,9.8) and (2.4,10) .. (3,10) .. controls (3.6,10) and (4,10.2) .. (4,11) .. controls (4,11.8) and (4.4,12) .. (5,12) .. controls (5.6,12) and (6,12.2) .. (6,13) -- (6,14)};
\wire{generator-2-2-0-pos}{(4,0) -- (4,1)(6,5) -- (6,7) .. controls (6,7.8) and (5,7.6) .. (5,8) .. controls (5,8.4) and (6,8.2) .. (6,9) -- (6,11)(4,13) -- (4,14)};
}
\end{scope}
\end{tikzpicture}

%% file: figures/reidemeister_v1/regular10.tikz
\begin{tikzpicture}
\definecolor{generator-3-2-0-pos}{RGB}{142, 68, 173}
\definecolor{generator-1-2-0-pos}{RGB}{192, 57, 43}
\definecolor{generator-2-2-0-pos}{RGB}{243, 156, 18}
\definecolor{generator-0-0-1-pos}{RGB}{246, 245, 244}

\newcommand{\wire}[2]{
  \ifdefined\recolor\draw[color=\recolor, line width=10pt]\else\draw[color=#1, line width=5pt]\fi #2;
}
\newcommand{\clipped}[3]{
\begin{scope}
  \newcommand{\recolor}{#1}
  \clip#3;
  #2
\end{scope}
}

\begin{scope}[transparency group]
% Background surfaces
\fill[generator-0-0-1-pos] (0,0) -- (8,0) -- (8,16) -- (0,16) -- (0,0);
\newcommand{\layer}[1]{
  \clipped{generator-0-0-1-pos}{#1}{(0,0) -- (8,0) -- (8,16) -- (0,16) -- (0,0)}
  #1
}

% Wire layers
\wire{generator-3-2-0-pos}{(6,1) .. controls (6,1.8) and (5.2,2) .. (4,2) .. controls (2.8,2) and (2,2.2) .. (2,3)};
\layer{
\wire{generator-2-2-0-pos}{(4,1) -- (4,3) .. controls (4,3.8) and (4.4,4) .. (5,4) .. controls (5.6,4) and (6,4.2) .. (6,5)(6,13) .. controls (6,13.8) and (5.6,14) .. (5,14) .. controls (4.4,14) and (4,14.2) .. (4,15)};
\wire{generator-3-2-0-pos}{(2,5) .. controls (2,5.8) and (2.4,6) .. (3,6) .. controls (3.6,6) and (4,6.2) .. (4,7) .. controls (4,7.8) and (4.4,8) .. (5,8) .. controls (5.6,8) and (6,8.2) .. (6,9) .. controls (6,9.8) and (5.6,10) .. (5,10) .. controls (4.4,10) and (4,10.2) .. (4,11) .. controls (4,11.8) and (3.6,12) .. (3,12) .. controls (2.4,12) and (2,12.2) .. (2,13)};
}
\layer{
\wire{generator-3-2-0-pos}{(6,0) -- (6,1)(2,3) -- (2,5)(2,13) -- (2,16)};
\wire{generator-1-2-0-pos}{(2,0) -- (2,1) .. controls (2,1.8) and (2.8,2) .. (4,2) .. controls (5.2,2) and (6,2.2) .. (6,3) .. controls (6,3.8) and (5.6,4) .. (5,4) .. controls (4.4,4) and (4,4.2) .. (4,5) .. controls (4,5.8) and (3.6,6) .. (3,6) .. controls (2.4,6) and (2,6.2) .. (2,7) -- (2,11) .. controls (2,11.8) and (2.4,12) .. (3,12) .. controls (3.6,12) and (4,12.2) .. (4,13) .. controls (4,13.8) and (4.4,14) .. (5,14) .. controls (5.6,14) and (6,14.2) .. (6,15) -- (6,16)};
\wire{generator-2-2-0-pos}{(4,0) -- (4,1)(6,5) -- (6,7) .. controls (6,7.8) and (5.6,8) .. (5,8) .. controls (4.4,8) and (4,8.2) .. (4,9) .. controls (4,9.8) and (4.4,10) .. (5,10) .. controls (5.6,10) and (6,10.2) .. (6,11) -- (6,13)(4,15) -- (4,16)};
}
\end{scope}
\end{tikzpicture}

%% file: figures/reidemeister_v1/regular11.tikz
\begin{tikzpicture}
\definecolor{generator-3-2-0-pos}{RGB}{142, 68, 173}
\definecolor{generator-1-2-0-pos}{RGB}{192, 57, 43}
\definecolor{generator-2-2-0-pos}{RGB}{243, 156, 18}
\definecolor{generator-0-0-1-pos}{RGB}{246, 245, 244}

\newcommand{\wire}[2]{
  \ifdefined\recolor\draw[color=\recolor, line width=10pt]\else\draw[color=#1, line width=5pt]\fi #2;
}
\newcommand{\clipped}[3]{
\begin{scope}
  \newcommand{\recolor}{#1}
  \clip#3;
  #2
\end{scope}
}

\begin{scope}[transparency group]
% Background surfaces
\fill[generator-0-0-1-pos] (0,0) -- (8,0) -- (8,14) -- (0,14) -- (0,0);
\newcommand{\layer}[1]{
  \clipped{generator-0-0-1-pos}{#1}{(0,0) -- (8,0) -- (8,14) -- (0,14) -- (0,0)}
  #1
}

% Wire layers
\wire{generator-3-2-0-pos}{(6,1) .. controls (6,1.8) and (4,1.6) .. (4,2) .. controls (4,2.4) and (2,2.2) .. (2,3)};
\layer{
\wire{generator-2-2-0-pos}{(4,1) -- (4,2) .. controls (4,2.4) and (6,2.2) .. (6,3)(6,11) .. controls (6,11.8) and (5.6,12) .. (5,12) .. controls (4.4,12) and (4,12.2) .. (4,13)};
\wire{generator-3-2-0-pos}{(2,3) .. controls (2,3.8) and (2.4,4) .. (3,4) .. controls (3.6,4) and (4,4.2) .. (4,5) .. controls (4,5.8) and (4.4,6) .. (5,6) .. controls (5.6,6) and (6,6.2) .. (6,7) .. controls (6,7.8) and (5.6,8) .. (5,8) .. controls (4.4,8) and (4,8.2) .. (4,9) .. controls (4,9.8) and (3.6,10) .. (3,10) .. controls (2.4,10) and (2,10.2) .. (2,11)};
}
\layer{
\wire{generator-3-2-0-pos}{(6,0) -- (6,1)(2,11) -- (2,14)};
\wire{generator-1-2-0-pos}{(2,0) -- (2,1) .. controls (2,1.8) and (4,1.6) .. (4,2) -- (4,3) .. controls (4,3.8) and (3.6,4) .. (3,4) .. controls (2.4,4) and (2,4.2) .. (2,5) -- (2,9) .. controls (2,9.8) and (2.4,10) .. (3,10) .. controls (3.6,10) and (4,10.2) .. (4,11) .. controls (4,11.8) and (4.4,12) .. (5,12) .. controls (5.6,12) and (6,12.2) .. (6,13) -- (6,14)};
\wire{generator-2-2-0-pos}{(4,0) -- (4,1)(6,3) -- (6,5) .. controls (6,5.8) and (5.6,6) .. (5,6) .. controls (4.4,6) and (4,6.2) .. (4,7) .. controls (4,7.8) and (4.4,8) .. (5,8) .. controls (5.6,8) and (6,8.2) .. (6,9) -- (6,11)(4,13) -- (4,14)};
}
\end{scope}
\end{tikzpicture}

%% file: figures/reidemeister_v1/regular12.tikz
\begin{tikzpicture}
\definecolor{generator-3-2-0-pos}{RGB}{142, 68, 173}
\definecolor{generator-1-2-0-pos}{RGB}{192, 57, 43}
\definecolor{generator-2-2-0-pos}{RGB}{243, 156, 18}
\definecolor{generator-0-0-1-pos}{RGB}{246, 245, 244}

\newcommand{\wire}[2]{
  \ifdefined\recolor\draw[color=\recolor, line width=10pt]\else\draw[color=#1, line width=5pt]\fi #2;
}
\newcommand{\clipped}[3]{
\begin{scope}
  \newcommand{\recolor}{#1}
  \clip#3;
  #2
\end{scope}
}

\begin{scope}[transparency group]
% Background surfaces
\fill[generator-0-0-1-pos] (0,0) -- (8,0) -- (8,12) -- (0,12) -- (0,0);
\newcommand{\layer}[1]{
  \clipped{generator-0-0-1-pos}{#1}{(0,0) -- (8,0) -- (8,12) -- (0,12) -- (0,0)}
  #1
}

% Wire layers
\wire{generator-3-2-0-pos}{(6,1) .. controls (6,1.8) and (4,1.6) .. (4,2) -- (4,3)};
\layer{
\wire{generator-2-2-0-pos}{(4,1) -- (4,2) .. controls (4,2.4) and (6,2.2) .. (6,3)(6,9) .. controls (6,9.8) and (5.6,10) .. (5,10) .. controls (4.4,10) and (4,10.2) .. (4,11)};
\wire{generator-3-2-0-pos}{(4,3) .. controls (4,3.8) and (4.4,4) .. (5,4) .. controls (5.6,4) and (6,4.2) .. (6,5) .. controls (6,5.8) and (5.6,6) .. (5,6) .. controls (4.4,6) and (4,6.2) .. (4,7) .. controls (4,7.8) and (3.6,8) .. (3,8) .. controls (2.4,8) and (2,8.2) .. (2,9)};
}
\layer{
\wire{generator-3-2-0-pos}{(6,0) -- (6,1)(2,9) -- (2,12)};
\wire{generator-1-2-0-pos}{(2,0) -- (2,1) .. controls (2,1.8) and (4,1.6) .. (4,2) .. controls (4,2.4) and (2,2.2) .. (2,3) -- (2,7) .. controls (2,7.8) and (2.4,8) .. (3,8) .. controls (3.6,8) and (4,8.2) .. (4,9) .. controls (4,9.8) and (4.4,10) .. (5,10) .. controls (5.6,10) and (6,10.2) .. (6,11) -- (6,12)};
\wire{generator-2-2-0-pos}{(4,0) -- (4,1)(6,3) .. controls (6,3.8) and (5.6,4) .. (5,4) .. controls (4.4,4) and (4,4.2) .. (4,5) .. controls (4,5.8) and (4.4,6) .. (5,6) .. controls (5.6,6) and (6,6.2) .. (6,7) -- (6,9)(4,11) -- (4,12)};
}
\end{scope}
\end{tikzpicture}

%% file: figures/reidemeister_v1/regular13.tikz
\begin{tikzpicture}
\definecolor{generator-3-2-0-pos}{RGB}{142, 68, 173}
\definecolor{generator-1-2-0-pos}{RGB}{192, 57, 43}
\definecolor{generator-2-2-0-pos}{RGB}{243, 156, 18}
\definecolor{generator-0-0-1-pos}{RGB}{246, 245, 244}

\newcommand{\wire}[2]{
  \ifdefined\recolor\draw[color=\recolor, line width=10pt]\else\draw[color=#1, line width=5pt]\fi #2;
}
\newcommand{\clipped}[3]{
\begin{scope}
  \newcommand{\recolor}{#1}
  \clip#3;
  #2
\end{scope}
}

\begin{scope}[transparency group]
% Background surfaces
\fill[generator-0-0-1-pos] (0,0) -- (8,0) -- (8,10) -- (0,10) -- (0,0);
\newcommand{\layer}[1]{
  \clipped{generator-0-0-1-pos}{#1}{(0,0) -- (8,0) -- (8,10) -- (0,10) -- (0,0)}
  #1
}

% Wire layers
\wire{generator-3-2-0-pos}{(6,1) .. controls (6,1.8) and (4,1.6) .. (4,2) .. controls (4,2.4) and (6,2.2) .. (6,3)};
\layer{
\wire{generator-2-2-0-pos}{(4,1) -- (4,3)(6,7) .. controls (6,7.8) and (5.6,8) .. (5,8) .. controls (4.4,8) and (4,8.2) .. (4,9)};
\wire{generator-3-2-0-pos}{(6,3) .. controls (6,3.8) and (5.6,4) .. (5,4) .. controls (4.4,4) and (4,4.2) .. (4,5) .. controls (4,5.8) and (3.6,6) .. (3,6) .. controls (2.4,6) and (2,6.2) .. (2,7)};
}
\layer{
\wire{generator-3-2-0-pos}{(6,0) -- (6,1)(2,7) -- (2,10)};
\wire{generator-1-2-0-pos}{(2,0) -- (2,1) .. controls (2,1.8) and (4,1.6) .. (4,2) .. controls (4,2.4) and (2,2.2) .. (2,3) -- (2,5) .. controls (2,5.8) and (2.4,6) .. (3,6) .. controls (3.6,6) and (4,6.2) .. (4,7) .. controls (4,7.8) and (4.4,8) .. (5,8) .. controls (5.6,8) and (6,8.2) .. (6,9) -- (6,10)};
\wire{generator-2-2-0-pos}{(4,0) -- (4,1)(4,3) .. controls (4,3.8) and (4.4,4) .. (5,4) .. controls (5.6,4) and (6,4.2) .. (6,5) -- (6,7)(4,9) -- (4,10)};
}
\end{scope}
\end{tikzpicture}

%% file: figures/reidemeister_v1/regular14.tikz
\begin{tikzpicture}
\definecolor{generator-3-2-0-pos}{RGB}{142, 68, 173}
\definecolor{generator-1-2-0-pos}{RGB}{192, 57, 43}
\definecolor{generator-2-2-0-pos}{RGB}{243, 156, 18}
\definecolor{generator-0-0-1-pos}{RGB}{246, 245, 244}

\newcommand{\wire}[2]{
  \ifdefined\recolor\draw[color=\recolor, line width=10pt]\else\draw[color=#1, line width=5pt]\fi #2;
}
\newcommand{\clipped}[3]{
\begin{scope}
  \newcommand{\recolor}{#1}
  \clip#3;
  #2
\end{scope}
}

\begin{scope}[transparency group]
% Background surfaces
\fill[generator-0-0-1-pos] (0,0) -- (8,0) -- (8,10) -- (0,10) -- (0,0);
\newcommand{\layer}[1]{
  \clipped{generator-0-0-1-pos}{#1}{(0,0) -- (8,0) -- (8,10) -- (0,10) -- (0,0)}
  #1
}

% Wire layers
\wire{generator-2-2-0-pos}{(6,7) .. controls (6,7.8) and (5.6,8) .. (5,8) .. controls (4.4,8) and (4,8.2) .. (4,9)};
\wire{generator-3-2-0-pos}{(6,1) .. controls (6,1.8) and (5,1.6) .. (5,2) .. controls (5,2.4) and (6,2.2) .. (6,3) .. controls (6,3.8) and (5.6,4) .. (5,4) .. controls (4.4,4) and (4,4.2) .. (4,5) .. controls (4,5.8) and (3.6,6) .. (3,6) .. controls (2.4,6) and (2,6.2) .. (2,7)};
\layer{
\wire{generator-3-2-0-pos}{(6,0) -- (6,1)(2,7) -- (2,10)};
\wire{generator-1-2-0-pos}{(2,0) -- (2,5) .. controls (2,5.8) and (2.4,6) .. (3,6) .. controls (3.6,6) and (4,6.2) .. (4,7) .. controls (4,7.8) and (4.4,8) .. (5,8) .. controls (5.6,8) and (6,8.2) .. (6,9) -- (6,10)};
\wire{generator-2-2-0-pos}{(4,0) -- (4,1) .. controls (4,1.8) and (5,1.6) .. (5,2) .. controls (5,2.4) and (4,2.2) .. (4,3) .. controls (4,3.8) and (4.4,4) .. (5,4) .. controls (5.6,4) and (6,4.2) .. (6,5) -- (6,7)(4,9) -- (4,10)};
}
\end{scope}
\end{tikzpicture}

%% file: figures/reidemeister_v1/regular15.tikz
\begin{tikzpicture}
\definecolor{generator-3-2-0-pos}{RGB}{142, 68, 173}
\definecolor{generator-1-2-0-pos}{RGB}{192, 57, 43}
\definecolor{generator-2-2-0-pos}{RGB}{243, 156, 18}
\definecolor{generator-0-0-1-pos}{RGB}{246, 245, 244}

\newcommand{\wire}[2]{
  \ifdefined\recolor\draw[color=\recolor, line width=10pt]\else\draw[color=#1, line width=5pt]\fi #2;
}
\newcommand{\clipped}[3]{
\begin{scope}
  \newcommand{\recolor}{#1}
  \clip#3;
  #2
\end{scope}
}

\begin{scope}[transparency group]
% Background surfaces
\fill[generator-0-0-1-pos] (0,0) -- (8,0) -- (8,8) -- (0,8) -- (0,0);
\newcommand{\layer}[1]{
  \clipped{generator-0-0-1-pos}{#1}{(0,0) -- (8,0) -- (8,8) -- (0,8) -- (0,0)}
  #1
}

% Wire layers
\wire{generator-2-2-0-pos}{(6,5) .. controls (6,5.8) and (5.6,6) .. (5,6) .. controls (4.4,6) and (4,6.2) .. (4,7)};
\wire{generator-3-2-0-pos}{(6,1) .. controls (6,1.8) and (5.6,2) .. (5,2) .. controls (4.4,2) and (4,2.2) .. (4,3) .. controls (4,3.8) and (3.6,4) .. (3,4) .. controls (2.4,4) and (2,4.2) .. (2,5)};
\layer{
\wire{generator-3-2-0-pos}{(6,0) -- (6,1)(2,5) -- (2,8)};
\wire{generator-1-2-0-pos}{(2,0) -- (2,3) .. controls (2,3.8) and (2.4,4) .. (3,4) .. controls (3.6,4) and (4,4.2) .. (4,5) .. controls (4,5.8) and (4.4,6) .. (5,6) .. controls (5.6,6) and (6,6.2) .. (6,7) -- (6,8)};
\wire{generator-2-2-0-pos}{(4,0) -- (4,1) .. controls (4,1.8) and (4.4,2) .. (5,2) .. controls (5.6,2) and (6,2.2) .. (6,3) -- (6,5)(4,7) -- (4,8)};
}
\end{scope}
\end{tikzpicture}

%% file: figures/reidemeister_v2/regular0.tikz
\begin{tikzpicture}
\definecolor{generator-3-2-0-pos}{RGB}{142, 68, 173}
\definecolor{generator-1-2-0-pos}{RGB}{192, 57, 43}
\definecolor{generator-2-2-0-pos}{RGB}{243, 156, 18}
\definecolor{generator-0-0-1-pos}{RGB}{246, 245, 244}

\newcommand{\wire}[2]{
  \ifdefined\recolor\draw[color=\recolor, line width=10pt]\else\draw[color=#1, line width=5pt]\fi #2;
}
\newcommand{\clipped}[3]{
\begin{scope}
  \newcommand{\recolor}{#1}
  \clip#3;
  #2
\end{scope}
}

\begin{scope}[transparency group]
% Background surfaces
\fill[generator-0-0-1-pos] (0,0) -- (8,0) -- (8,8) -- (0,8) -- (0,0);
\newcommand{\layer}[1]{
  \clipped{generator-0-0-1-pos}{#1}{(0,0) -- (8,0) -- (8,8) -- (0,8) -- (0,0)}
  #1
}

% Wire layers
\wire{generator-2-2-0-pos}{(4,1) .. controls (4,1.8) and (3.6,2) .. (3,2) .. controls (2.4,2) and (2,2.2) .. (2,3)};
\wire{generator-3-2-0-pos}{(6,3) .. controls (6,3.8) and (5.6,4) .. (5,4) .. controls (4.4,4) and (4,4.2) .. (4,5) .. controls (4,5.8) and (3.6,6) .. (3,6) .. controls (2.4,6) and (2,6.2) .. (2,7)};
\layer{
\wire{generator-3-2-0-pos}{(6,0) -- (6,3)(2,7) -- (2,8)};
\wire{generator-1-2-0-pos}{(2,0) -- (2,1) .. controls (2,1.8) and (2.4,2) .. (3,2) .. controls (3.6,2) and (4,2.2) .. (4,3) .. controls (4,3.8) and (4.4,4) .. (5,4) .. controls (5.6,4) and (6,4.2) .. (6,5) -- (6,8)};
\wire{generator-2-2-0-pos}{(4,0) -- (4,1)(2,3) -- (2,5) .. controls (2,5.8) and (2.4,6) .. (3,6) .. controls (3.6,6) and (4,6.2) .. (4,7) -- (4,8)};
}
\end{scope}
\end{tikzpicture}

%% file: figures/reidemeister_v2/regular1.tikz
\begin{tikzpicture}
\definecolor{generator-3-2-0-pos}{RGB}{142, 68, 173}
\definecolor{generator-1-2-0-pos}{RGB}{192, 57, 43}
\definecolor{generator-2-2-0-pos}{RGB}{243, 156, 18}
\definecolor{generator-0-0-1-pos}{RGB}{246, 245, 244}

\newcommand{\wire}[2]{
  \ifdefined\recolor\draw[color=\recolor, line width=10pt]\else\draw[color=#1, line width=5pt]\fi #2;
}
\newcommand{\clipped}[3]{
\begin{scope}
  \newcommand{\recolor}{#1}
  \clip#3;
  #2
\end{scope}
}

\begin{scope}[transparency group]
% Background surfaces
\fill[generator-0-0-1-pos] (0,0) -- (8,0) -- (8,10) -- (0,10) -- (0,0);
\newcommand{\layer}[1]{
  \clipped{generator-0-0-1-pos}{#1}{(0,0) -- (8,0) -- (8,10) -- (0,10) -- (0,0)}
  #1
}

% Wire layers
\wire{generator-2-2-0-pos}{(4,3) .. controls (4,3.8) and (3.6,4) .. (3,4) .. controls (2.4,4) and (2,4.2) .. (2,5)};
\wire{generator-3-2-0-pos}{(6,1) .. controls (6,1.8) and (5,1.6) .. (5,2) .. controls (5,2.4) and (6,2.2) .. (6,3)(6,5) .. controls (6,5.8) and (5.6,6) .. (5,6) .. controls (4.4,6) and (4,6.2) .. (4,7) .. controls (4,7.8) and (3.6,8) .. (3,8) .. controls (2.4,8) and (2,8.2) .. (2,9)};
\layer{
\wire{generator-3-2-0-pos}{(6,0) -- (6,1)(6,3) -- (6,5)(2,9) -- (2,10)};
\wire{generator-1-2-0-pos}{(2,0) -- (2,3) .. controls (2,3.8) and (2.4,4) .. (3,4) .. controls (3.6,4) and (4,4.2) .. (4,5) .. controls (4,5.8) and (4.4,6) .. (5,6) .. controls (5.6,6) and (6,6.2) .. (6,7) -- (6,10)};
\wire{generator-2-2-0-pos}{(4,0) -- (4,1) .. controls (4,1.8) and (5,1.6) .. (5,2) .. controls (5,2.4) and (4,2.2) .. (4,3)(2,5) -- (2,7) .. controls (2,7.8) and (2.4,8) .. (3,8) .. controls (3.6,8) and (4,8.2) .. (4,9) -- (4,10)};
}
\end{scope}
\end{tikzpicture}

%% file: figures/reidemeister_v2/regular2.tikz
\begin{tikzpicture}
\definecolor{generator-3-2-0-pos}{RGB}{142, 68, 173}
\definecolor{generator-1-2-0-pos}{RGB}{192, 57, 43}
\definecolor{generator-2-2-0-pos}{RGB}{243, 156, 18}
\definecolor{generator-0-0-1-pos}{RGB}{246, 245, 244}

\newcommand{\wire}[2]{
  \ifdefined\recolor\draw[color=\recolor, line width=10pt]\else\draw[color=#1, line width=5pt]\fi #2;
}
\newcommand{\clipped}[3]{
\begin{scope}
  \newcommand{\recolor}{#1}
  \clip#3;
  #2
\end{scope}
}

\begin{scope}[transparency group]
% Background surfaces
\fill[generator-0-0-1-pos] (0,0) -- (8,0) -- (8,8) -- (0,8) -- (0,0);
\newcommand{\layer}[1]{
  \clipped{generator-0-0-1-pos}{#1}{(0,0) -- (8,0) -- (8,8) -- (0,8) -- (0,0)}
  #1
}

% Wire layers
\wire{generator-3-2-0-pos}{(6,1) .. controls (6,1.8) and (4,1.6) .. (4,2) .. controls (4,2.4) and (6,2.2) .. (6,3)};
\layer{
\wire{generator-2-2-0-pos}{(4,1) -- (4,2) .. controls (4,2.4) and (2,2.2) .. (2,3)};
\wire{generator-3-2-0-pos}{(6,3) .. controls (6,3.8) and (5.6,4) .. (5,4) .. controls (4.4,4) and (4,4.2) .. (4,5) .. controls (4,5.8) and (3.6,6) .. (3,6) .. controls (2.4,6) and (2,6.2) .. (2,7)};
}
\layer{
\wire{generator-3-2-0-pos}{(6,0) -- (6,1)(2,7) -- (2,8)};
\wire{generator-1-2-0-pos}{(2,0) -- (2,1) .. controls (2,1.8) and (4,1.6) .. (4,2) -- (4,3) .. controls (4,3.8) and (4.4,4) .. (5,4) .. controls (5.6,4) and (6,4.2) .. (6,5) -- (6,8)};
\wire{generator-2-2-0-pos}{(4,0) -- (4,1)(2,3) -- (2,5) .. controls (2,5.8) and (2.4,6) .. (3,6) .. controls (3.6,6) and (4,6.2) .. (4,7) -- (4,8)};
}
\end{scope}
\end{tikzpicture}

%% file: figures/reidemeister_v2/regular3.tikz
\begin{tikzpicture}
\definecolor{generator-3-2-0-pos}{RGB}{142, 68, 173}
\definecolor{generator-1-2-0-pos}{RGB}{192, 57, 43}
\definecolor{generator-2-2-0-pos}{RGB}{243, 156, 18}
\definecolor{generator-0-0-1-pos}{RGB}{246, 245, 244}

\newcommand{\wire}[2]{
  \ifdefined\recolor\draw[color=\recolor, line width=10pt]\else\draw[color=#1, line width=5pt]\fi #2;
}
\newcommand{\clipped}[3]{
\begin{scope}
  \newcommand{\recolor}{#1}
  \clip#3;
  #2
\end{scope}
}

\begin{scope}[transparency group]
% Background surfaces
\fill[generator-0-0-1-pos] (0,0) -- (8,0) -- (8,6) -- (0,6) -- (0,0);
\newcommand{\layer}[1]{
  \clipped{generator-0-0-1-pos}{#1}{(0,0) -- (8,0) -- (8,6) -- (0,6) -- (0,0)}
  #1
}

% Wire layers
\wire{generator-3-2-0-pos}{(6,1) .. controls (6,1.8) and (5.2,2) .. (4,2) -- (4,3)};
\layer{
\wire{generator-2-2-0-pos}{(4,1) -- (4,2) .. controls (2.8,2) and (2,2.2) .. (2,3)};
\wire{generator-3-2-0-pos}{(4,3) .. controls (4,3.8) and (3.6,4) .. (3,4) .. controls (2.4,4) and (2,4.2) .. (2,5)};
}
\layer{
\wire{generator-3-2-0-pos}{(6,0) -- (6,1)(2,5) -- (2,6)};
\wire{generator-1-2-0-pos}{(2,0) -- (2,1) .. controls (2,1.8) and (2.8,2) .. (4,2) .. controls (5.2,2) and (6,2.2) .. (6,3) -- (6,6)};
\wire{generator-2-2-0-pos}{(4,0) -- (4,1)(2,3) .. controls (2,3.8) and (2.4,4) .. (3,4) .. controls (3.6,4) and (4,4.2) .. (4,5) -- (4,6)};
}
\end{scope}
\end{tikzpicture}

%% file: figures/reidemeister_v2/regular4.tikz
\begin{tikzpicture}
\definecolor{generator-3-2-0-pos}{RGB}{142, 68, 173}
\definecolor{generator-1-2-0-pos}{RGB}{192, 57, 43}
\definecolor{generator-2-2-0-pos}{RGB}{243, 156, 18}
\definecolor{generator-0-0-1-pos}{RGB}{246, 245, 244}

\newcommand{\wire}[2]{
  \ifdefined\recolor\draw[color=\recolor, line width=10pt]\else\draw[color=#1, line width=5pt]\fi #2;
}
\newcommand{\clipped}[3]{
\begin{scope}
  \newcommand{\recolor}{#1}
  \clip#3;
  #2
\end{scope}
}

\begin{scope}[transparency group]
% Background surfaces
\fill[generator-0-0-1-pos] (0,0) -- (8,0) -- (8,4) -- (0,4) -- (0,0);
\newcommand{\layer}[1]{
  \clipped{generator-0-0-1-pos}{#1}{(0,0) -- (8,0) -- (8,4) -- (0,4) -- (0,0)}
  #1
}

% Wire layers
\wire{generator-3-2-0-pos}{(6,1) .. controls (6,1.8) and (5.2,2) .. (4,2) .. controls (2.8,2) and (2,2.2) .. (2,3)};
\layer{
\wire{generator-2-2-0-pos}{(4,1) -- (4,3)};
}
\layer{
\wire{generator-3-2-0-pos}{(6,0) -- (6,1)(2,3) -- (2,4)};
\wire{generator-1-2-0-pos}{(2,0) -- (2,1) .. controls (2,1.8) and (2.8,2) .. (4,2) .. controls (5.2,2) and (6,2.2) .. (6,3) -- (6,4)};
\wire{generator-2-2-0-pos}{(4,0) -- (4,1)(4,3) -- (4,4)};
}
\end{scope}
\end{tikzpicture}

%% file: figures/reidemeister_v2/regular5.tikz
\begin{tikzpicture}
\definecolor{generator-3-2-0-pos}{RGB}{142, 68, 173}
\definecolor{generator-1-2-0-pos}{RGB}{192, 57, 43}
\definecolor{generator-2-2-0-pos}{RGB}{243, 156, 18}
\definecolor{generator-0-0-1-pos}{RGB}{246, 245, 244}

\newcommand{\wire}[2]{
  \ifdefined\recolor\draw[color=\recolor, line width=10pt]\else\draw[color=#1, line width=5pt]\fi #2;
}
\newcommand{\clipped}[3]{
\begin{scope}
  \newcommand{\recolor}{#1}
  \clip#3;
  #2
\end{scope}
}

\begin{scope}[transparency group]
% Background surfaces
\fill[generator-0-0-1-pos] (0,0) -- (8,0) -- (8,6) -- (0,6) -- (0,0);
\newcommand{\layer}[1]{
  \clipped{generator-0-0-1-pos}{#1}{(0,0) -- (8,0) -- (8,6) -- (0,6) -- (0,0)}
  #1
}

% Wire layers
\wire{generator-3-2-0-pos}{(4,3) -- (4,4) .. controls (2.8,4) and (2,4.2) .. (2,5)};
\layer{
\wire{generator-3-2-0-pos}{(6,1) .. controls (6,1.8) and (5.6,2) .. (5,2) .. controls (4.4,2) and (4,2.2) .. (4,3)};
\wire{generator-2-2-0-pos}{(6,3) .. controls (6,3.8) and (5.2,4) .. (4,4) -- (4,5)};
}
\layer{
\wire{generator-3-2-0-pos}{(6,0) -- (6,1)(2,5) -- (2,6)};
\wire{generator-1-2-0-pos}{(2,0) -- (2,3) .. controls (2,3.8) and (2.8,4) .. (4,4) .. controls (5.2,4) and (6,4.2) .. (6,5) -- (6,6)};
\wire{generator-2-2-0-pos}{(4,0) -- (4,1) .. controls (4,1.8) and (4.4,2) .. (5,2) .. controls (5.6,2) and (6,2.2) .. (6,3)(4,5) -- (4,6)};
}
\end{scope}
\end{tikzpicture}

%% file: figures/reidemeister_v2/regular6.tikz
\begin{tikzpicture}
\definecolor{generator-3-2-0-pos}{RGB}{142, 68, 173}
\definecolor{generator-1-2-0-pos}{RGB}{192, 57, 43}
\definecolor{generator-2-2-0-pos}{RGB}{243, 156, 18}
\definecolor{generator-0-0-1-pos}{RGB}{246, 245, 244}

\newcommand{\wire}[2]{
  \ifdefined\recolor\draw[color=\recolor, line width=10pt]\else\draw[color=#1, line width=5pt]\fi #2;
}
\newcommand{\clipped}[3]{
\begin{scope}
  \newcommand{\recolor}{#1}
  \clip#3;
  #2
\end{scope}
}

\begin{scope}[transparency group]
% Background surfaces
\fill[generator-0-0-1-pos] (0,0) -- (8,0) -- (8,8) -- (0,8) -- (0,0);
\newcommand{\layer}[1]{
  \clipped{generator-0-0-1-pos}{#1}{(0,0) -- (8,0) -- (8,8) -- (0,8) -- (0,0)}
  #1
}

% Wire layers
\wire{generator-3-2-0-pos}{(2,5) .. controls (2,5.8) and (4,5.6) .. (4,6) .. controls (4,6.4) and (2,6.2) .. (2,7)};
\layer{
\wire{generator-3-2-0-pos}{(6,1) .. controls (6,1.8) and (5.6,2) .. (5,2) .. controls (4.4,2) and (4,2.2) .. (4,3) .. controls (4,3.8) and (3.6,4) .. (3,4) .. controls (2.4,4) and (2,4.2) .. (2,5)};
\wire{generator-2-2-0-pos}{(6,5) .. controls (6,5.8) and (4,5.6) .. (4,6) -- (4,7)};
}
\layer{
\wire{generator-3-2-0-pos}{(6,0) -- (6,1)(2,7) -- (2,8)};
\wire{generator-1-2-0-pos}{(2,0) -- (2,3) .. controls (2,3.8) and (2.4,4) .. (3,4) .. controls (3.6,4) and (4,4.2) .. (4,5) -- (4,6) .. controls (4,6.4) and (6,6.2) .. (6,7) -- (6,8)};
\wire{generator-2-2-0-pos}{(4,0) -- (4,1) .. controls (4,1.8) and (4.4,2) .. (5,2) .. controls (5.6,2) and (6,2.2) .. (6,3) -- (6,5)(4,7) -- (4,8)};
}
\end{scope}
\end{tikzpicture}

%% file: figures/reidemeister_v2/regular7.tikz
\begin{tikzpicture}
\definecolor{generator-3-2-0-pos}{RGB}{142, 68, 173}
\definecolor{generator-1-2-0-pos}{RGB}{192, 57, 43}
\definecolor{generator-2-2-0-pos}{RGB}{243, 156, 18}
\definecolor{generator-0-0-1-pos}{RGB}{246, 245, 244}

\newcommand{\wire}[2]{
  \ifdefined\recolor\draw[color=\recolor, line width=10pt]\else\draw[color=#1, line width=5pt]\fi #2;
}
\newcommand{\clipped}[3]{
\begin{scope}
  \newcommand{\recolor}{#1}
  \clip#3;
  #2
\end{scope}
}

\begin{scope}[transparency group]
% Background surfaces
\fill[generator-0-0-1-pos] (0,0) -- (8,0) -- (8,8) -- (0,8) -- (0,0);
\newcommand{\layer}[1]{
  \clipped{generator-0-0-1-pos}{#1}{(0,0) -- (8,0) -- (8,8) -- (0,8) -- (0,0)}
  #1
}

% Wire layers
\wire{generator-2-2-0-pos}{(6,5) .. controls (6,5.8) and (5.6,6) .. (5,6) .. controls (4.4,6) and (4,6.2) .. (4,7)};
\wire{generator-3-2-0-pos}{(6,1) .. controls (6,1.8) and (5.6,2) .. (5,2) .. controls (4.4,2) and (4,2.2) .. (4,3) .. controls (4,3.8) and (3.6,4) .. (3,4) .. controls (2.4,4) and (2,4.2) .. (2,5)};
\layer{
\wire{generator-3-2-0-pos}{(6,0) -- (6,1)(2,5) -- (2,8)};
\wire{generator-1-2-0-pos}{(2,0) -- (2,3) .. controls (2,3.8) and (2.4,4) .. (3,4) .. controls (3.6,4) and (4,4.2) .. (4,5) .. controls (4,5.8) and (4.4,6) .. (5,6) .. controls (5.6,6) and (6,6.2) .. (6,7) -- (6,8)};
\wire{generator-2-2-0-pos}{(4,0) -- (4,1) .. controls (4,1.8) and (4.4,2) .. (5,2) .. controls (5.6,2) and (6,2.2) .. (6,3) -- (6,5)(4,7) -- (4,8)};
}
\end{scope}
\end{tikzpicture}

%% file: 2-anticolimits.tex
\section{Anticolimits} \label{sec:anticolimits}

Let $\pos J$ be a poset, seen as a thin category.
Given a diagram $F : \pos J \to \cat C$, a cocone over $F$ can be equivalently defined as either one of:
\begin{enumerate}
\item a natural transformation $\kappa : F \Rightarrow \Delta_c$, or
\item a sink $\kappa : F|_{\max \pos J} \Rightarrow \Delta_c$ such that
\[
\begin{tikzcd}
& c & \\
F j \urar["\kappa_j"] && F k \ular["\kappa_k"'] \\
& F i \ular \urar &
\end{tikzcd}
\]
commutes for every $i \in \pos J$ and $j, k \in \Upper i$.
\end{enumerate}
While (1) is the standard definition, we will use (2) as our definition of cocone, because it requires the least amount of data.
A colimit is a cocone which is universal, i.e.\ any cocone factors uniquely through it.
Now the problem we are interested in is the following: given a sink $\kappa$ over $\max \pos J$, find all diagrams over $\pos J$ for which $\kappa$ is a colimit.

\begin{definition}
Let $\pos J$ be a poset and consider a sink
\[
\kappa : X \Rightarrow \Delta_c : \max \pos J \to \cat C
\]
We define a \emph{$\pos J$-anticocone} of $\kappa$ to be a diagram $A : \pos J \to \cat C$ which extends $X$ such that $\kappa$ is a cocone over it.
We say that an anticocone is an \emph{anticolimit} if furthermore it makes $\kappa$ into a colimit.
\end{definition}

For example, consider the following poset $\pos J$ (i.e.\ a span):
\[
\begin{tikzcd}
\bullet && \bullet \\
& \bullet \ular \urar &
\end{tikzcd}
\]
This has two maximal elements, so a sink over $\max \pos J$ is a cospan
\[
\begin{tikzcd}
& C & \\
A \urar["f"] && B \ular["g"']
\end{tikzcd}
\]
The $\pos J$-anticocones are spans $(p : D \to A, \, q : D \to B)$ such that $f \circ p = g \circ q$,
and the $\pos J$-anticolimits, which we call \emph{antipushouts}, are spans making the following diagram into a pushout square:
\[
\begin{tikzcd}
& C \ar[dd, phantom, "\llcorner" rotate=45, very near start] & \\
A \urar["f"] && B \ular["g"'] \\
& D \ular[dashed, "p"] \urar[dashed, "q"'] & 
\end{tikzcd}
\]

\begin{example}
Consider the cospan in the category $\Ord$:
\[
\begin{tikzcd}
& \ord 1 & \\
\ord 2 \urar && \ord 3 \ular
\end{tikzcd}
\]
Then the following spans are both antipushouts of the cospan:
\[
\begin{tikzcd}
& \ord 1 \ar[dd, phantom, "\llcorner" rotate=45, very near start] & \\
\ord 2 \urar && \ord 3 \ular \\
& \ord 4 \ular[dashed, "0001"] \urar[dashed, "0122"'] &
\end{tikzcd}
\qquad
\begin{tikzcd}
& \ord 1 \ar[dd, phantom, "\llcorner" rotate=45, very near start] & \\
\ord 2 \urar && \ord 3 \ular \\
& \ord 4 \ular[dashed, "0111"] \urar[dashed, "0012"'] &
\end{tikzcd}
\]
Here, we write $a_1 \dots a_n : \ord n \to \ord m$ for the map that sends $i$ to $a_i$.
\end{example}

In general, anticolimits may not exist, and even if they do, they are not necessarily unique.
They do {not} satisfy a universal property.
There is an easy way to test if a sink cannot have anticolimits.

\begin{proposition}
\label{prop:jointly-epic}
If $\kappa$ has an anticolimit, then it is jointly epic.
\end{proposition}

\begin{proof}
Colimit inclusions are guaranteed to be jointly epic.
Also, $\set{\kappa_i}_{i \in \pos J}$ is jointly epic iff $\set{\kappa_i}_{i \in \max \pos J}$ is jointly epic, because each $\kappa_i$ for $i$ non-maximal factorises through some $\kappa_j$ for $j$ maximal.
\end{proof}

Anticocones and anticolimits form categories.
First, we need to define the category of extensions of one functor along another.

\begin{definition}
An \emph{extension} of a functor $F : \cat C \to \cat D$ along another functor $G : \cat C \to \cat E$ is a functor $H : \cat E \to \cat D$ such that $H \circ G = F$:
\[
\begin{tikzcd}
\cat C \dar["G"'] \rar["F"] & \cat D \\
\cat E \urar[dashed, "H"'] &
\end{tikzcd}
\]
The category $\Ext_G(F)$ of extensions of $F$ along $G$ is the subcategory of $[\cat E, \cat D]$ obtained by the following pullback in $\Cat$:
\[
\begin{tikzcd}
\Ext_G(F) \arrow[dr, phantom, "\lrcorner", very near start] \dar \rar[tail] & {[\cat E, \cat D]} \dar["(-) \circ G"] \\
\ord 1 \rar["F"', tail] & {[\cat C, \cat D]}
\end{tikzcd}
\]
In particular, the morphisms are given by natural transformations $\eta : H_1 \to H_2$ such that $\eta \circ G = \id_F$.
If $i : \cat C \hookrightarrow \cat E$ is a subcategory inclusion, we write $\Ext_\cat E(F)$ instead of $\Ext_i(F)$; the morphisms here are natural transformations whose components on $\cat C$ are trivial.
\end{definition}

\begin{definition}
The categories $\Acc_\pos J(\kappa)$ of anticocones and $\Acl_\pos J(\kappa)$ of anticolimits are defined to be full subcategories of $\Ext_\pos J(X)$:
\[
\Acl_\pos J(\kappa) \hookrightarrow \Acc_\pos J(\kappa) \hookrightarrow \Ext_\pos J(X)
\]
\end{definition}

\paragraph{Sieves and cosieves}
Recall that a \emph{sieve} in a category $\cat C$ is a full subcategory of $\cat C$ closed under precomposition with morphisms in $\cat C$, according to~\cite[Definition 6.2.2.1]{Lur09}.
Dually, a \emph{cosieve} is a full subcategory closed under postcomposition with morphisms in $\cat C$.

We have the following results about the categories of anticocones and anticolimits.
The full proofs are included in the appendix.

\begin{lemma}
\label{lma:acc-sieve}
$\Acc_\pos J(\kappa)$ is a sieve in $\Ext_\pos J(X)$.
\end{lemma}

\begin{proof}
Let $\eta : A \to B$ be a morphism in $\Ext_\pos J(X)$ and suppose that $B$ is an anticocone of $\kappa$.
Then the following commutes for every $i \in \pos J$ and $j, k \in \Upper i$ since $\kappa$ is a cocone over $B$ and by naturality of $\eta$:
\[
\begin{tikzcd}
& c & \\
B j \urar["\kappa_j"] && B k \ular["\kappa_k"'] \\
& B i \ular \urar & \\
A j \ar[uu, equal] && A k \ar[uu, equal] \\
& A i \ular \urar \ar[uu, "\eta_i"'] &
\end{tikzcd}
\]
Therefore $\kappa$ is a cocone over $A$, which implies that $A \in \Acc_\pos J(\kappa)$.
\end{proof}

\begin{lemma}
\label{lma:acl-cosieve}
$\Acl_\pos J(\kappa)$ is a cosieve in $\Acc_\pos J(X)$.
\end{lemma}

\begin{proof}
Let $\eta : A \to B$ be a morphism in $\Acc_\pos J(X)$ and suppose that $A$ is an anticolimit of $\kappa$.
Now let $\lambda : X \Rightarrow \Delta_d$ be a cocone over $B$.
Then the following commutes for every $i \in \pos J$ and $j, k \in \Upper i$:
\[
\begin{tikzcd}
& d & \\
B j \urar["\lambda_j"] && B k \ular["\lambda_k"'] \\
& B i \ular \urar & \\
A j \ar[uu, equal] && A k \ar[uu, equal] \\
& A i \ular \urar \ar[uu, "\eta_i"'] &
\end{tikzcd}
\]
Hence $\lambda$ is also a cocone over $A$, so there is universal morphism $u : c \to d$ such that $\lambda = \Delta_u \circ \kappa$.
Therefore $\kappa$ is a colimit of $B$.
\end{proof}

\begin{lemma}
\label{lma:acl-pointwise-epi}
If $\eta : A \to B$ is a pointwise epimorphism in $\Ext_\pos J(X)$ and $B$ is an anticolimit of $\kappa$, then $A$ is also an anticolimit of $\kappa$.
\end{lemma}

\begin{proof}
We already have that $\kappa$ is a cocone over $A$ by \Cref{lma:acc-sieve}.
Now let $\lambda : X \Rightarrow \Delta_d$ be a cocone over $A$.
Then the outer diagram below commutes for every $i \in \pos J$ and $j, k \in \Upper i$:
\[
\begin{tikzcd}
& d & \\
B j \urar["\kappa_j"] && B k \ular["\kappa_k"'] \\
& B i \ular \urar & \\
A j \ar[uu, equal] && A k \ar[uu, equal] \\
& A i \ular \urar \ar[uu, two heads, "\eta_i"'] &
\end{tikzcd}
\]
Since $\eta_i$ is epic, the following diagram also commutes:
\[
\begin{tikzcd}
& d & \\
B j \urar["\kappa_j"] && B k \ular["\kappa_k"'] \\
& B i \ular \urar &
\end{tikzcd}
\]
Hence $\lambda$ is also a cocone over $B$, so there is universal morphism $u : c \to d$ such that $\lambda = \Delta_u \circ \kappa$.
Therefore $\kappa$ is a colimit of $A$.
\end{proof}

\subsection{Pullback construction}
\label{subsec:pullback}

For any poset $\pos J$ and sink $\kappa$, we can define a $\pos J$-anticocone $\Pi_\pos J(\kappa)$ by a pullback construction.
When this exists, we show that it is terminal in $\Acc_\pos J(\kappa)$, and that $\kappa$ has $\pos J$-anticolimits iff $\Pi_\pos J(\kappa)$ is an anticolimit.
Therefore, $\Pi_\pos J(\kappa)$ can be considered the ``most general'' $\pos J$-anticolimit of $\kappa$, and can be used to verify the existence of anticolimits.

\begin{definition}
For a poset $\pos J$ and a sink $\kappa$, we define
\[
\Pi_\pos J(\kappa) : \pos J \to \cat C
\]
by the following construction:
\begin{itemize}
\item every object $i$ is mapped to the pullback of $\set{\kappa_j \mid j \in \Upper i}$;
\item every morphism $i \leq j$ is mapped to the universal morphism between the respective pullbacks that exists since $\Upper j \subseteq \Upper i$.
\end{itemize}
If any of these pullbacks fail to exist, then $\Pi_\pos J(\kappa)$ does not exist.
\end{definition}

Note that this is an anticocone of $\kappa$ by the universal property of pullbacks.
To illustrate how this works, consider a simple example.

\begin{example}
Let $\pos J = \mathcal{P}^+(\ord 3)^\mathrm{op}$ and let $\kappa$ be a sink:
\[
\begin{tikzcd}
& C & \\
X_0 \urar & X_1 \uar & X_2 \ular
\end{tikzcd}
\]
The associated anticocone $\Pi_\pos J(\kappa)$ is given by the following diagram:
\[
\begin{tikzcd}
& C & \\
X_0 \urar & X_1 \uar & X_2 \ular \\
X_0 \times_C X_1 \uar[dashed] \urar[dashed] & X_0 \times_C X_2 \ular[dashed] \urar[dashed] & X_1 \times_C X_2 \ular[dashed] \uar[dashed] \\
& X_0 \times_C X_1 \times_C X_2 \ular[dashed] \uar[dashed] \urar[dashed] & 
\end{tikzcd}
\]
\end{example}

\begin{lemma}
\label{lma:terminal-anticocone}
If $\Pi_\pos J(\kappa)$ exists, it is the terminal object in $\Acc_\pos J(\kappa)$.
\end{lemma}

\begin{proof}
If $A$ is a $\pos J$-anticocone of $\kappa$, then $A i$ is a cone over $\set{\kappa_j}_{j \in \Upper i}$ for every $i \in \pos J$, so we have a universal morphism into the pullback $u_i: A i \to \Pi_\pos J(\kappa)(i)$.
Now the following commutes whenever $i \leq j$ due to the universal property of pullbacks:
\[
\begin{tikzcd}
A i \dar["u_i"'] \rar & A j \dar["u_j"] \\
\Pi_\pos J(\kappa)(i) \rar & \Pi_\pos J(\kappa)(j)
\end{tikzcd}
\]
As a result, we have a natural transformation
\[
u : A \Rightarrow \Pi_\pos J(\kappa).
\]
This is trivial on $\max \pos J$, so it is a morphism in $\Acc_\pos J(X)$.
Moreover, it is unique since the components $u_i$ are unique by construction.
\end{proof}

\begin{corollary}
If $\Pi_\pos J(\kappa)$ exists, then the category of anticocones $\Acc_\pos J(\kappa)$ is isomorphic to the slice category $\Ext_\pos J(X) / \Pi_\pos J(\kappa)$, with the isomorphism given by the forgetful functor of the slice category.
\end{corollary}

\begin{proof}
This follows from \Cref{lma:acc-sieve,lma:terminal-anticocone}.
\end{proof}

\begin{theorem}
\label{thm:canonical-anticolimit}
If $\Pi_\pos J(\kappa)$ exists, either one of the following holds:
\begin{itemize}
\item $\kappa$ has no $\pos J$-anticolimits, or
\item $\Pi_\pos J(\kappa)$ is an anticolimit of $\kappa$.
\end{itemize}
\end{theorem}

\begin{proof}
If $\kappa$ has an anticolimit $A$, we have a map $A \to \Pi_\pos J(\kappa)$ in $\Acc_\pos J(\kappa)$ by \Cref{lma:terminal-anticocone}, so $\Pi_\pos J(\kappa)$ is an anticolimit by \Cref{lma:acl-cosieve}.
\end{proof}

Hence, the category $\Acl_\pos J(\kappa)$ is isomorphic to a full subcategory of $\Ext_\pos J(X) / \Pi_\pos J(\kappa)$ which is empty or contains the terminal object.
Also if $\Pi_\pos J(\kappa)$ happens to have a colimit, we have a universal map:
\[
\colim \Pi_\pos J(\kappa) \to c
\]
Then $\kappa$ has $\pos J$-anticolimits iff this morphism is an isomorphism.

\begin{corollary}
Given a pair of morphisms $(f, g)$ with the same codomain which have a pullback, the following are equivalent:
\begin{itemize}
\item the morphisms $(f, g)$ have an antipushout, and
\item the following is simultaneously both a pullback and a pushout:\footnotemark{}
\end{itemize}
\[
\begin{tikzcd}[column sep=small]
& C \ar[dd, phantom, "\llcorner" rotate=45, very near start] & \\
A \urar["f"] && B \ular["g"'] \\
& A \times_C B \ular["\pi_A"] \urar["\pi_B"'] \ar[uu, phantom, "\urcorner" rotate=45, very near start] &
\end{tikzcd}
\]
\end{corollary}

\footnotetext{Such squares are called \emph{bicartesian} or \emph{pulation squares}.}

For example, consider the following cospan in $\Ord$ for $n > 1$:
\[
\begin{tikzcd}
& \ord 1 & \\
\ord 0 \urar && \ord n \ular
\end{tikzcd}
\]
By the previous corollary, it does \emph{not} have any antipushouts, as the pullback is $\ord 0$, and the pushout of the pullback is then $\ord n$.
Note that this is despite the cospan being jointly epic (as per \Cref{prop:jointly-epic}).

\subsection{Change of shape}

Any monotone map $f : \pos I \to \pos J$ induces a functor
\[
(-) \circ f : [\pos J, \cat C] \to [\pos I, \cat C]
\]
Under certain conditions, this interacts nicely with anticolimits.

\begin{lemma}
Suppose that we have functors
\[
\begin{tikzcd}
\cat C_1 \dar["G_1"'] \rar["\alpha"] & \cat C_2 \dar["G_2"] \rar["F"] & \cat D \\
\cat E_1 \rar["\beta"'] & \cat E_2 \urar[dashed, "H"']
\end{tikzcd}
\]
\begin{enumerate}
\item $\beta$ preserves extensions, i.e.\ if $H$ is an extension of $F$ along $G_2$ then $H \circ \beta$ is an extension of $F \circ \alpha$ along $G_1$:
\[
\begin{tikzcd}[column sep=large]
\Ext_{G_2}(F) \rar["(-) \circ \beta"] \dar[hook] & \Ext_{G_1}(F \circ \alpha) \dar[hook] \\
{[\cat E_2, \cat D]} \rar["(-) \circ \beta"'] & {[\cat E_1, \cat D]}
\end{tikzcd}
\]
\item If $\alpha$ is epic, $\beta$ reflects extensions, i.e.\ if $H \circ \beta$ is an extension of $F \circ \alpha$ along $G_1$ then $H$ is an extension of $F$ along $G_2$:
\[
\begin{tikzcd}[column sep=large]
\Ext_{G_2}(F) \ar[dr, phantom, "\lrcorner" very near start] \rar["(-) \circ \beta"] \dar[hook] & \Ext_{G_1}(F \circ \alpha) \dar[hook] \\
{[\cat E_2, \cat D]} \rar["(-) \circ \beta"'] & {[\cat E_1, \cat D]}
\end{tikzcd}
\]
\end{enumerate}
\end{lemma}

\begin{proof}
(1) is straightforward.
For (2), we have that
\[
H \circ G_2 \circ \alpha = H \circ \beta \circ G_1 = F \circ \alpha \implies H \circ G_2 = F
\]
\end{proof}

\begin{definition}
Let $f : \pos I \to \pos J$ be a monotone map that preserves maximal elements.
Then we define the following terminology:
\begin{itemize}
\item $f$ \emph{preserves} $\pos J$-anticocones of $\kappa$ if:
\[
\forall A \in \Ext_\pos J(X) \ldotp A \in \Acc_\pos J(\kappa) \implies A \circ f \in \Acc_\pos I(\kappa \circ f)
\]
\item $f$ \emph{reflects} $\pos J$-anticocones of $\kappa$ if:
\[
\forall A \in \Ext_\pos J(X) \ldotp A \circ f \in \Acc_\pos I(F \circ \kappa) \implies A \in \Acc_\pos J(\kappa)
\]
\item $f$ \emph{preserves} $\pos J$-anticolimits of $\kappa$ if:
\[
\forall A \in \Acc_\pos J(\kappa) \ldotp A \in \Acl_\pos J(\kappa) \implies A \circ f \in \Acl_\pos I(\kappa \circ f)
\]
\item $f$ \emph{reflects} $\pos J$-anticolimits of $\kappa$ if:
\[
\forall A \in \Acc_\pos J(\kappa) \ldotp A \circ f \in \Acl_\pos I(F \circ \kappa) \implies  A \in \Acl_\pos J(\kappa)
\]
\end{itemize}
\end{definition}

\begin{definition}
A monotone map $f : \pos I \to \pos J$ is \emph{fair} if
\begin{itemize}
\item it preserves and lifts maximal elements:
\[
\begin{tikzcd}
\max \pos I \dar[hook] \rar[two heads, "f"] & \max \pos J \dar[hook] \\
\pos I \rar["f"'] & \pos J
\end{tikzcd}
\]
\item it lifts pairwise lower bounds of maximal elements:
\begin{center}
\input{figures/lift.tikz}
\end{center}
That is, for all $j \in \pos J$ and $i_1, i_2 \in \max \pos J$ with $j \leq f i_1$ and $j \leq f i_2$, there exists $i_0 \in \pos I$ with $i_0 \leq i_1$ and $i_0 \leq i_2$ and $j = f i_0$.
\end{itemize}
\end{definition}

\paragraph{Final functors}
Recall that a functor $F : \cat C \to \cat D$ is \emph{final} if for any functor $G : \cat D \to \cat E$, there is a natural isomorphism
\[
\colim GF \cong \colim G.
\]
Also recall that $F$ is final iff for every $d \in \cat D$, the comma category $d / F$ is connected, i.e.\ any two objects are connected by a zigzag of morphisms.
Note that if $f : \pos I \to \pos J$ is a monotone map, the comma category $j / f$ for $j \in \pos J$ is just the sub-poset $f^{-1}(\upper j) \subseteq \pos I$.

\begin{theorem}
\label{thm:change-of-shape}
Let $f : \pos I \to \pos J$ be a monotone map.
\begin{enumerate}
\item If $f$ is fair, it preserves and reflects anticocones:
\[
\begin{tikzcd}[column sep=large]
\Acc_\pos J(\kappa) \ar[dr, phantom, "\lrcorner" very near start] \rar["(-) \circ f"] \dar[hook] & \Acc_\pos I(\kappa \circ f) \dar[hook] \\
\Ext_\pos J(X) \rar["(-) \circ f"'] & \Ext_\pos I(X \circ f) 
\end{tikzcd}
\]
\item If $f$ is also final, it preserves and reflects anticolimits:
\[
\begin{tikzcd}[column sep=large]
\Acl_\pos J(\kappa) \ar[dr, phantom, "\lrcorner" very near start] \rar["(-) \circ f"] \dar[hook] & \Acl_\pos I(\kappa \circ f) \dar[hook] \\
\Acc_\pos J(\kappa) \rar["(-) \circ f"'] & \Acc_\pos I(\kappa \circ f) 
\end{tikzcd}
\]
\end{enumerate}
\end{theorem}

\begin{proof}
Note that if $f$ is fair, then the class of diagrams
\begin{equation*}
\begin{tikzcd}
& c & \\
A j_1 \urar["\kappa_{j_1}"] && A j_2 \ular["\kappa_{j_2}"'] \\
& A j_0 \ular \urar &
\end{tikzcd}
\end{equation*}
for $j_0 \in \pos J$ and $j_1, j_2 \in \Upper j_0$, is equivalent to the following class of diagrams, for $i_0 \in \pos I$ and $i_1, i_2 \in \Upper i_0$:
\begin{equation*}
\begin{tikzcd}
& c & \\
A f i_1 \urar["\kappa_{f i_1}"] && A f i_2 \ular["\kappa_{f i_2}"'] \\
& A f i_0 \ular \urar &
\end{tikzcd}
\end{equation*}
In particular, the first class is contained in the second class since $f$ lifts maximal elements and lower bounds of maximal elements, and vice-versa since $f$ preserves maximal elements.
Hence, $f$ preserves and reflects anticocones.
Finally, (2) follows immediately from $f$ being final.
\end{proof}

Therefore, given a fair and final monotone map $f : \pos I \to \pos J$, we can compute $\pos J$-anticolimits in terms of $\pos I$-anticolimits, which are assumed to be simpler.
In particular, given an $\pos I$-anticolimit $A$, we can construct an anticolimit in $\pos J$ by extending along $f$ as follows:
\[
\begin{tikzcd}
\pos I \dar["f"'] \rar["A"] & \cat C \\
\pos J \urar[dashed]
\end{tikzcd}
\]
Note that a particular $\pos I$-anticolimit may fail to extend.
However, any extension gives a valid $\pos J$-anticolimit and \emph{all} $\pos J$-anticolimits can be obtained in this way, for any fixed choice of $f : \pos I \to \pos J$.

\subsubsection{Removing non-sinking edges}

Here we introduce a technique that, for a given poset $\pos J$, may be used to construct a poset $\pos I$ and a fair and final monotone map $f : \pos I \to \pos J$, such that $\pos I$ is simpler.
In particular, $\pos I$ will have fewer arrows.
Recall that in a poset $\pos J$, we say that $a$ is \emph{covered} by $b$ if $a < b$ and there is no $c$ such that $a < c < b$.

\begin{definition}
Let $\pos J$ be a poset with $a$ covered by $b$.
We define
\[
\pos J \setminus \set{a < b}
\]
to be the poset obtained from $\pos J$ by removing the relation $a < b$.
\end{definition}

For example, consider removing the red arrow here:
\[
\begin{tikzcd}
& \bullet & \\
\bullet \urar && \bullet \ular \\
& \bullet \ular[red] \urar &
\end{tikzcd}  
\quad\leadsto\quad
\begin{tikzcd}
& \bullet & \\
\bullet \urar && \bullet \ular \\
& \bullet  \urar &
\end{tikzcd}
\]

\begin{lemma}
\label{lma:non-sinking-edge}
If $b$ is non-maximal, this map is fair and final: 
\[
\pos J \setminus \set{a < b} \to \pos J
\]
\end{lemma}

\begin{proof}
Removing $a < b$ cannot remove maximal elements and may only create one maximal element, that being $a$ if $\upper a = \set{a, b}$.
Since $b$ is non-maximal, the maximal elements are the same.

The map does not change the lower bounds of maximal elements, because $a < b$ does not participate in any lower bounds of maximal elements, since $b$ is non-maximal.
Hence the map is fair. To show that the map is final, we need to show that the subset $\upper i$ is connected in $\pos J \setminus \set{a < b}$ for every $i \in \pos J$.
We have two cases:
\begin{enumerate}
\item If $i \neq a$, then $i$ is below every element in $\upper i$.
\item If $i = a$, then $a$ is below every element in $\upper a$ except $b$, but $a$ and $b$ are both below some $c \in \upper a$ since $b$ is non-maximal.
\end{enumerate}
Therefore, $\upper i$ is connected in both cases, so the map is final.
\end{proof}

Note that an extension of the form
\[
\begin{tikzcd}
\pos J \setminus \set{i < j} \dar \rar["A"] & \cat C \\
\pos J \urar[dashed]
\end{tikzcd}
\]
consists of finding a morphism $A i \to A j$ that commutes with the rest of the diagram, which is a tractable problem in many categories.

\subsection{Hypergraph-like posets}

Consider the class of posets in which every non-maximal element is minimal.
These can be described in terms of hypergraphs.

\begin{definition}
A hypergraph $H = (V, E, i)$ consists of:
\begin{itemize}
\item a set $V$ of vertices,
\item a set $E$ of hyperedges, and
\item an incidence function $i : E \to \mathcal{P}^+(V)$.
\end{itemize}
\end{definition}

\begin{proposition}
There is a bijection between:
\begin{itemize}
\item hypergraphs $H = (V, E, i)$, and
\item posets $\pos J$ such that every non-maximal element is minimal.
\end{itemize}
\end{proposition}

\begin{proof}
Every hypergraph has an associated poset $V \sqcup E$ where $e \leq v$ iff $v \in i(e)$.
Conversely, every poset $\pos J$ such that every non-maximal element is minimal is the associated poset of a hypergraph:
\[
V \coloneqq \max \pos J \qquad E \coloneqq \pos J \setminus \max \pos J \qquad i(e) \coloneqq \Upper e
\]
\end{proof}

\begin{figure}
\centering
\input{figures/hypergraph.tikz}
\caption{Correspondence between hypergraphs and posets.}
\end{figure}
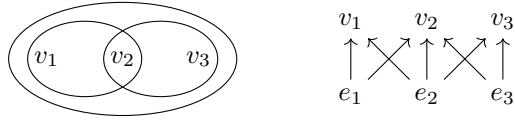

We say that such a poset is \emph{hypergraph-like}.
Note that a diagram $A$ over a hypergraph-like poset can be described as follows:
\begin{itemize}
\item for every vertex $v$, an object $A_v$, and
\item for every hyperedge $e$, a source $\set{A_e \to A_v}_{v \in i(e)}$.
\end{itemize}
i.e.\ a collection of objects and generalised relations between them.

As a corollary of \Cref{lma:non-sinking-edge}, we can reduce any poset $\pos J$ to a hypergraph-like poset $\widehat{\pos J}$ such that $\widehat{\pos J} \to \pos J$ is both fair and final.

\subsection{Anticolimits of sets}

We will now show how to compute anticolimits in $\Set$.
As shown in the previous section, it suffices to consider hypergraph-like posets, so fix a hypergraph $H = (V, E, i)$ and a sink $\set{f_v : A_v \to C}_{v \in V}$. We construct anticolimits according to the following scheme:
\begin{enumerate}
\item For every $e \in E$, compute the pullback
\[
\Pi_e \coloneqq \mathrm{pullback}(\set{f_v : A_v \to C}_{v \in i(e)})
\]
\item Choose a family of maps
\[
\set{g_e : A_e \to \Pi_e}_{e \in E}
\]
\item Compute the colimit of the maps
\[
\set{A_e \xrightarrow{g_e} \Pi_e \xrightarrow{\pi_v} A_v}_{e \in E, \, v \in i(e)}
\]
Recall that this is given by the following quotient:
\[
Q \coloneqq \coprod_{v \in V} A_v / \sim
\]
where $\sim$ is the smallest equivalence relation such that
\[
\pi_{v_1}(g_e(x)) \sim \pi_{v_2}(g_e(x))
\]
whenever $e \in E$, and $v_1, v_2 \in i(e)$, and $x \in A_e$.
\item Check that the universal map $u : Q \to C$ is an isomorphism.
\end{enumerate}

\subsection{Change of category}

Any functor $F : \cat C \to \cat D$ induces a functor
\[
F \circ (-) : [\pos J, \cat C] \to [\pos J, \cat D]
\]
Under certain conditions, this interacts nicely with anticolimits.

\begin{lemma}
Suppose that we have functors
\[
\begin{tikzcd}
\cat C \dar["G"'] \rar["F"] & \cat D_1 \rar["\alpha"] & \cat D_2 \\
\cat E \urar[dashed, "H"']
\end{tikzcd}
\]
\begin{enumerate}
\item $\alpha$ preserves extensions, i.e.\ if $H$ is an extension of $F$ along $G$ then $\alpha \circ H$ is an extension of $\alpha \circ F$ along $G$:
\[
\begin{tikzcd}[column sep=large]
\Ext_G(F) \rar["\alpha \circ (-)"] \dar[hook] & \Ext_G(\alpha \circ F) \dar[hook] \\
{[\cat E, \cat D_1]} \rar["\alpha \circ (-)"'] & {[\cat E, \cat D_2]}
\end{tikzcd}
\]
\item If $\alpha$ is monic, it reflects extensions, i.e.\ if $\alpha \circ H$ is an extension of $\alpha \circ F$ along $G$ then $H$ is an extension of $F$ along $G$:
\[
\begin{tikzcd}[column sep=large]
\Ext_G(F) \ar[dr, phantom, "\lrcorner" very near start] \rar[tail, "\alpha \circ (-)"] \dar[hook] & \Ext_G(\alpha \circ F) \dar[hook] \\
{[\cat E, \cat D_1]} \rar[tail, "\alpha \circ (-)"'] & {[\cat E, \cat D_2]}
\end{tikzcd}
\]
\end{enumerate}
\end{lemma}

\begin{proof}
(1) is straightforward.
(2) follows from $\alpha$ being monic.
\end{proof}

\begin{definition}
Let $F : \cat C \to \cat D$ be a functor.
\begin{itemize}
\item $F$ \emph{preserves} $\pos J$-anticocones of $\kappa$ if:
\[
\forall A \in \Ext_\pos J(X) \ldotp A \in \Acc_\pos J(\kappa) \implies F \circ A \in \Acc_\pos J(F \circ \kappa)
\]
\item $F$ \emph{reflects} $\pos J$-anticocones of $\kappa$ if:
\[
\forall A \in \Ext_\pos J(X) \ldotp F \circ A \in \Acc_\pos J(F \circ \kappa) \implies A \in \Acc_\pos J(\kappa)
\]
\item $F$ \emph{preserves} $\pos J$-anticolimits of $\kappa$ if:
\[
\forall A \in \Acc_\pos J(\kappa) \ldotp A \in \Acl_\pos J(\kappa) \implies F \circ A \in \Acl_\pos J(F \circ \kappa)
\]
\item $F$ \emph{reflects} $\pos J$-anticolimits of $\kappa$ if:
\[
\forall A \in \Acc_\pos J(\kappa) \ldotp F \circ A \in \Acl_\pos J(F \circ \kappa) \implies  A \in \Acl_\pos J(\kappa)
\]
\end{itemize}
\end{definition}

\begin{definition}
A faithful functor $F : \cat C \to \cat D$ is \emph{directly nice} if it preserves and reflects $\pos J$-colimits, and \emph{indirectly nice} if it preserves $\pos J$-colimits, and $\cat C$ has pullbacks and $F$ preserves them.
\end{definition}

\begin{theorem}
\label{thm:change-of-category}
Let $F : \cat C \to \cat D$ be a functor.
\begin{enumerate}
\item If $F$ is faithful, it preserves and reflects anticocones:
\[
\begin{tikzcd}[column sep=large]
\Acc_\pos J(\kappa) \ar[dr, phantom, "\lrcorner" very near start] \rar["F \circ (-)"] \dar[hook] & \Acc_\pos J(F \circ \kappa) \dar[hook] \\
\Ext_\pos J(X) \rar["F \circ (-)"'] & \Ext_\pos J(F \circ X)
\end{tikzcd}
\]
\item If $F$ is directly nice, it preserves and reflects anticolimits, and similarly if $F$ is indirectly nice and $\kappa$ already has anticolimits:
\[
\begin{tikzcd}[column sep=large]
\Acl_\pos J(\kappa) \ar[dr, phantom, "\lrcorner" very near start] \rar["F \circ (-)"] \dar[hook] & \Acl_\pos I(F \circ \kappa) \dar[hook] \\
\Acc_\pos J(\kappa) \rar["F \circ (-)"'] & \Acc_\pos I(F \circ \kappa) 
\end{tikzcd}
\]
\end{enumerate}
\end{theorem}

\begin{proof}
Assume that $F$ is faithful, and let $A$ be an extension of~$X$.
Now for every $i \in \pos J$ and $j, k \in \Upper i$, the following diagram:
\[
\begin{tikzcd}
& c & \\
A j \urar["\kappa_j"] && A k \ular["\kappa_k"'] \\
& A i \ular \urar &
\end{tikzcd}
\]
commutes iff the following diagram commutes:
\[
\begin{tikzcd}
& F c & \\
F A j \urar["F \kappa_j"] && F A k \ular["F \kappa_k"'] \\
& F A i \ular \urar &
\end{tikzcd} 
\]
since $F$ is faithful.
Hence, $F$ preserves and reflects anticocones.

Clearly, if $F$ is directly nice, it preserves and reflects anticolimits, and if $F$ is indirectly nice, it preserves anticolimits.
Suppose that $\kappa$ has anticolimits.
Then by \Cref{thm:canonical-anticolimit}, $\Pi_\pos J(\kappa)$ is an anticolimit.
Let $A \to \Pi_\pos J(\kappa)$ be an anticocone such that $F \circ \kappa$ is universal over $F \circ A$.
Now if $\lambda : X \Rightarrow \Delta_d$ is a cocone over $A$, then $F \circ \lambda$ is a cocone over $F \circ A$, so there is a universal map $u' : F c \to F d$ such that
\[
F \circ \lambda = \Delta_{u'} \circ F \circ \kappa
\]
Then $F \circ \lambda$ is a cocone over $\Pi_\pos J(F \circ \kappa)$.
Since $F$ preserves pullbacks,
\[
\Pi_\pos J(F \circ \kappa) = F \circ \Pi_\pos J(\kappa)
\]
Since $F$ is faithful, $\lambda$ must be a cocone over $\Pi_\pos J(\kappa)$.
Finally, since $\kappa$ is universal over $\Pi_\pos J(\kappa)$, there is a universal map $u : c \to d$ with
\[
\lambda = \Delta_u \circ \kappa.
\]
Therefore, $\kappa$ is universal over $A$, so $F$ reflects anticolimits.
\end{proof}

Therefore, if a functor $F : \cat C \to \cat D$ is directly or indirectly nice, we can compute anticolimits in the category $\cat C$ in terms of anticolimits in the category $\cat D$.
In particular, given an anticolimit $A$ in $\cat D$, we can construct an anticolimit in $\cat C$ by finding a lift of the form
\[
\begin{tikzcd}
\max \pos J \dar[hook] \rar["X"] & \cat C \dar["F"] \\
\pos J \rar["A"'] \urar[dashed] & \cat D  
\end{tikzcd}
\]
Note that a particular anticolimit in $\cat D$ may fail to have such a lift.
However, any lift is a valid anticolimit in $\cat C$ and \emph{all} anticolimits in $\cat C$ can be obtained in this way, for any fixed choice of $F : \cat C \to \cat D$.
If $F$ is indirectly nice, this construction only works if, a priori, the sink had anticolimits in $\cat C$, but that is easy to check using \Cref{thm:canonical-anticolimit}.
Finally, if $F$ is monic, then it suffices to find a lift of the form
\[
\begin{tikzcd}
& \cat C \dar[tail, "F"] \\
\pos J \rar["A"'] \urar[dashed] & \cat D  
\end{tikzcd}
\]
i.e.\ any lift automatically extends $X$, and the lift is unique if it exists.

\subsection{Anticolimits of finite ordinals}

We will now show how to compute anticolimits in $\Ord$ in terms of anticolimits in $\Set$, using the result from the previous section.

\begin{lemma}
The inclusion $\Ord \hookrightarrow \Pos$ preserves colimits.
\end{lemma}

\begin{proof}
Suppose that we have a universal cocone in $\Ord$:
\[
\eta : F \Rightarrow \Delta_\ord n
\]
We want to show that it is also a universal in $\Pos$, so let $\lambda : F \Rightarrow \Delta_P$ be a cocone for a poset $P$.
Now $P$ has a linear extension $f : P \hookrightarrow L$ by the order extension principle~\cite{Szp30}.
Then we have a cocone in $\Ord$:
\[
\lambda_L \coloneqq \Delta_f \circ \lambda : F \Rightarrow \Delta_L
\]
Hence, there exists a universal map $u_L : \ord n \to L$ with $\lambda_L = \Delta_{u_L} \circ \eta$.
Also by the order extension principle, $P$ is the intersection of all its linear extensions $L$, so we can construct a map $u : \ord n \to P$ such that $\lambda = \Delta_u \circ \eta$.
Therefore, $\eta$ is indeed a universal cocone in $\Pos$.
\end{proof}

\begin{corollary}
The inclusion $\Ord \hookrightarrow \Pos$ is directly nice.
\end{corollary}

\begin{proof}
The functor is fully faithful, so it reflects colimits.
Also it preserves colimits by the previous lemma.
\end{proof}

This means that we can compute anticolimits in $\Ord$ in terms of anticolimits in $\Pos$ by lifting each anticolimit in $\Pos$ as follows:
\[
\begin{tikzcd}
& \Ord \dar[hook] \\
\pos J \rar["A"'] \urar[dashed] & \Pos
\end{tikzcd}
\]
This simply consists of checking whether the diagram $A$ lies in $\Ord$.

\begin{lemma}
The forgetful functor $\Pre \to \Set$ is indirectly nice.
\end{lemma}

\begin{proof}
Clearly it is faithful.
Also it preserves limits and colimits, since it has a left adjoint (sending each set to its discrete preorder\footnotemark{}) and a right adjoint (sending each set to its indiscrete preoder\footnotemark{}).
\end{proof}

\footnotetext{The \emph{discrete preorder} on a set $X$ is the preorder $\leq$ given by $x \leq x$ for all $x \in X$.}
\footnotetext{The \emph{indiscrete preorder} on a set $X$ is the preorder $\leq$ given by $x \leq y$ for all $x, y \in X$.}

This means that we can compute anticolimits in $\Pre$ in terms of anticolimits in $\Set$ by lifting each anticolimit in $\Set$ as follows:
\[
\begin{tikzcd}
\max \pos J \dar[hook] \rar["X"] & \Pre \dar \\
\pos J \rar["A"'] \urar[dashed] & \Set
\end{tikzcd}
\]
This consists of choosing preorders on $A i$ for $i$ non-maximal that are compatible with each other and the preorders on $X i$ for $i$ maximal.

Finally, consider the inclusion $\Pos \hookrightarrow \Pre$.
It does not preserve colimits, however it has a left adjoint which does:\footnotemark{}
\[
L : \Pre \to \Pos
\]
Hence, we can compute anticolimits in $\Pos$ in terms of anticolimits in $\Pre$ via the following procedure, for every sink $\kappa$ in $\Pos$:
\begin{enumerate}
\item Choose a lift $\widetilde{\kappa}$ of $\kappa$ along $L$.
\item Choose an anticolimit $A$ of $\widetilde{\kappa}$ in $\Pre$.
\item Therefore $L A$ is an anticolimit of $\kappa$ in $\Pos$.
\end{enumerate}
Note that \emph{all} anticolimits in $\Pos$ can be obtained in this way, for the following reason.
If $A$ is an anticolimit of $\kappa$ in $\Pos$, then $A$ has a colimit $\widetilde{\kappa}$ in $\Pre$, and moreover $L \widetilde{\kappa} = \kappa$ since $L$ preserves colimits.

\footnotetext{This sends every preorder $P$ to the poset $P / \sim$ given by $a \sim b$ iff $a \leq b$ and $b \leq a$.}

%% file: figures/lift.tikz
\begin{tikzpicture}[yscale=.8]

\node (I) at (0, 2, 1) {$\pos I$};
\node (J) at (0, 0, 1) {$\pos J$};

\draw[->] (I) to node[left] {$f$} (J);

\node (i0) at (2, 2, 1) {$\exists i_0$};
\node (i1) at (4, 2, 2) {$i_1$};
\node (i2) at (4, 2, 0) {$i_2$};

\draw[->, dashed] (i0) to (i1);
\draw[->, dashed] (i0) to (i2);

\node (j0) at (2, 0, 1) {$j$};
\node (j1) at (4, 0, 2) {$f i_1$};
\node (j2) at (4, 0, 0) {$f i_2$};

\draw[->] (j0) to (j1);
\draw[->] (j0) to (j2);

\end{tikzpicture}

%% file: figures/hypergraph.tikz
\begin{tikzpicture}[scale=0.5]

\node at (0, 1) {$v_1$};
\node at (2, 1) {$v_2$};
\node at (4, 1) {$v_3$};

\draw (1, 1) ellipse (1.5 and 1);
\draw (2, 1) ellipse (3 and 1.5);
\draw (3, 1) ellipse (1.5 and 1);

\node (v1) at (8, 2) {$v_1$};
\node (v2) at (10, 2) {$v_2$};
\node (v3) at (12, 2) {$v_3$};
\node (e1) at (8, 0) {$e_1$};
\node (e2) at (10, 0) {$e_2$};
\node (e3) at (12, 0) {$e_3$};

\draw[->] (e1) to (v1);
\draw[->] (e1) to (v2);
\draw[->] (e2) to (v1);
\draw[->] (e2) to (v2);
\draw[->] (e2) to (v3);
\draw[->] (e3) to (v2);
\draw[->] (e3) to (v3);

\end{tikzpicture}

%% file: 3-zigzags.tex
\section{Zigzag category} \label{sec:zigzags}

The zigzag category construction yields a simple inductive definition of the terms of the theory of associative $n$-categories, known as \emph{$n$-diagrams}.
Following Reutter and Vicary~\cite{RV19}, we review the basic elements of the theory. This section is mostly background, although some of the presentation is novel.

\begin{definition}
In a category $\cat C$, a \emph{zigzag} $X$ is a finite sequence of cospans:
\[
\begin{tikzcd}[column sep=small]
X(\reg 0) \rar & X(\sing 0) & X(\reg 1) \lar \rar & \cdots & X(\reg n) \lar
\end{tikzcd}
\]
This has \emph{length} $\abs{X} = n \geq 0$.
The objects $X(\sing i)$ for $i \in \ord n$ are called the \emph{singular objects}, and $X(\reg j)$ for $j \in \ord{n + 1}$ the \emph{regular objects}.
\end{definition}

To define a notion of morphisms between zigzags, we first recall the duality between ordinals and intervals due to Wraith~\cite{Wra93}:
\[
\Reg : \Ord \xrightarrow{\simeq} \Int^\mathrm{op}
\]
This maps each $\ord n$ to $\ord{n + 1}$ and each monotone map $f : \ord n \to \ord m$ in $\Ord$ to the monotone map $\Reg f : \ord{m + 1} \to \ord{n + 1}$ in $\Int$ given by
\[
(\Reg f)(i) \coloneqq \min(\set{j \in \ord n \mid f(j) \geq i} \cup \set{n}).
\]
This has a nice geometrical intuition, which is illustrated in \Cref{fig:interleaving}: the map $f$ going up the page is ``interleaved'' with the map $\Reg f$ going down the page.
We can obtain the definition of a zigzag map from this picture, by instantiating the arrows with actual morphisms.

\begin{figure}
\centering
\begin{tikzcd}[column sep=small]
\color{blue}{\bullet} \ar[dr, blue] & \color{red}{\bullet} & \color{blue}{\bullet} \ar[dr, blue] & \color{red}{\bullet} & \color{blue}{\bullet} \ar[dl, blue] & \color{red}{\bullet} & \color{blue}{\bullet} \ar[dr, blue] & \color{red}{\bullet} & \color{blue}{\bullet} \ar[dl, blue] \\
& \color{blue}{\bullet} & \color{red}{\bullet} \ar[ul, red] & \color{blue}{\bullet} & \color{red}{\bullet} \ar[ur, red] & \color{blue}{\bullet} & \color{red}{\bullet} \ar[ul, red] & \color{blue}{\bullet}
\end{tikzcd}
\caption{An interleaving of $\color{red} f : \ord 3 \to \ord 4$ in $\Ord$ going up, and $\color{blue} \Reg f : \ord 5 \to \ord 4$ in $\Int$ going down.}
\label{fig:interleaving}

\end{figure}

\begin{definition}
In a category $\cat C$, a \emph{zigzag map} $f : X \to Y$ between a zigzag $X$ of length $n$ and and a zigzag $Y$ of length $m$ consists of the following:
\begin{itemize}
\item a \emph{singular map} in $\Ord$,
\[
f_\mathsf{s} : \ord n \to \ord m
\]
with an implied \emph{regular map} in $\Int$,
\[
f_\mathsf{r} \coloneqq \Reg f_\mathsf{s} : \ord{m + 1} \to \ord{n + 1}
\]
\item for every $i \in \ord n$, a \emph{singular slice}
\[
f(\sing i) : X(\sing i) \to Y(\sing f_\mathsf{s}(i))
\]
\item for every $j \in \ord{m + 1}$, a \emph{regular slice}
\[
f(\reg j) : X(\reg f_\mathsf{r}(j)) \to Y(\reg j)
\]
\end{itemize}
These must satisfy the following conditions for every $i \in \ord m$:
\begin{itemize}
\item if $f_\mathsf{s}^{-1}(i)$ is empty, this diagram commutes:
\[
\begin{tikzcd}[column sep=small]
Y(\reg i) \rar & Y(\sing i) & Y(\reg i + 1) \lar \\
& X(\reg f_\mathsf{r}(i)) \ular["f(\reg i)"] \urar["f(\reg i + 1)"'] &
\end{tikzcd}
\]
\item if $f_\mathsf{s}^{-1}(i) = [a, b]$ is non-empty, these diagrams commute:
\[
\begin{tikzcd}[column sep=small]
Y(\reg i) \rar & Y(\sing i) \\
X(\reg a) \rar \uar["f(\reg i)"] & X(\sing a) \uar["f(\sing a)"']
\end{tikzcd}
\qquad
\begin{tikzcd}[column sep=small]
Y(\sing i) & Y(\reg i + 1) \lar \\
X(\sing b) \uar["f(\sing b)"] & X(\reg b + 1) \lar \uar["f(\reg i + 1)"']
\end{tikzcd}
\]
and the following diagram commutes for every $a \leq j < b$:
\[
\begin{tikzcd}[column sep=small]
& Y(\sing i) & \\
X(\sing j) \urar["f(\sing j)"] & X(\reg j + 1) \lar \rar & X(\sing j + 1) \ular["f(\sing j + 1)"']
\end{tikzcd}
\]
\end{itemize}
These unique composites are called \emph{diagonal slices} and are denoted by $f(\reg j, \sing i) : X(\reg j) \to Y(\sing i)$, where $i \in \ord m$ and $j \in f_\mathsf{r}([i, i + 1])$.
\end{definition}

\begin{figure}
\centering
\begin{tikzcd}[column sep=small]
Y(\reg 0) \rar & Y(\sing 0) & Y(\reg 1) \lar \rar & Y(\sing 1) & Y(\reg 2) \lar \rar & Y(\sing 2) & Y(\reg 3) \lar \rar & Y(\sing 3) & Y(\reg 4) \lar \\
& X(\reg 0) \rar \ar[ul, blue] & X(\sing 0) \ar[ul, red] & X(\reg 1) \lar \rar \ar[ul, blue] \ar[ur, blue] & X(\sing 1) \ar[ur, red] & X(\reg 2) \lar \rar & X(\sing 2) \ar[ul, red] & X(\reg 3) \lar \ar[ul, blue] \ar[ur, blue]
\end{tikzcd}
\caption{A zigzag map $f : X \to Y$ going up with the singular slices in red and regular slices in blue.}
\end{figure}
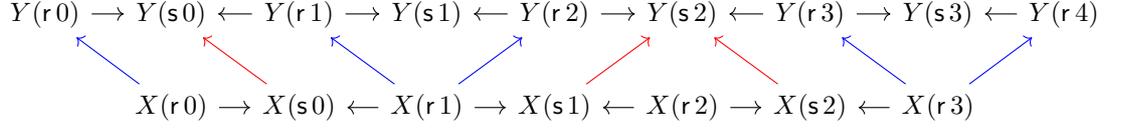

Given a category $\cat C$, we define its \emph{zigzag category} $\Zig(\cat C)$ to be the category whose objects are zigzags and whose morphisms are zigzag maps in $\cat C$.
This is functorial, i.e.\ we have a functor
\[
\Zig : \Cat \to \Cat
\]
If we iterate $n$ times, we obtain the \emph{$n$-fold zigzag category} $\Zig^n(\cat C)$.
 
Note that $\Zig(\ord 1)$ is isomorphic to $\Ord$, since a zigzag in the terminal category is uniquely determined by its length.
Since every category $\cat C$ has a unique functor $\cat C \to \ord 1$, this naturally induces a functor
\[
\pi : \Zig(\cat C) \to \Ord
\]
sending a zigzag to its length and a zigzag map to its singular map.

\paragraph{Globularity}
The original definition of zigzag maps~\cite{RV19} required all regular slices to be identities, a condition we drop here.
We call such zigzag maps \emph{globular}, and we write $\Zig_=(\cat C)$ for the wide subcategory of globular maps.

\begin{definition}
Let $X$ be a zigzag of length $n$.
Its \emph{restriction} $X[a, b]$, for $a, b \in \ord{n + 1}$ with $a \leq b$, is defined to be the zigzag restricted to the cospans between the regular objects $X(\reg a)$ and $X(\reg b)$.
\end{definition}

\begin{definition}
Let $f : X \to Y$ be a zigzag map with $Y$ of length $m$, and let $a, b \in \ord{m + 1}$ with $a < b$.
We define the \emph{restriction}
\[
f[a, b] : X[c, d] \to Y[a, b]
\]
for $[c, d] = f_\mathsf{r}([a, b])$, to be the zigzag map restricted to everything between $f(\reg a) : X(\reg c) \to Y(\reg a)$ and $f(\reg b) : X(\reg d) \to Y(\reg b)$.
\end{definition}

\subsection{Diagrams}

Let $\Sigma$ be a set of labels equipped with a dimension function
\[
\dim : \Sigma \to \mathbb{N}
\]
This induces a partial order given by $f < g$ iff $\dim f < \dim g$.

\begin{definition}
A \emph{$\Sigma$-typed $n$-diagram} is an object of $\Zig^n_=(\Sigma)$.
\end{definition}

We think of $\Sigma$ as encoding a signature and the objects of $\Zig^n_=(\Sigma)$ as $n$\-dimensional string diagrams in the free associative $n$\-category generated by that signature.
Not all $n$-diagrams are valid, and the full theory of associative $n$-categories has a \emph{typechecking} procedure for checking if an $n$-diagram is \emph{well-typed} with respect to a signature.
For a detailed description, see~\cite[Chapter 8]{Dor18} and~\cite[Section 7]{HRV22}.

\begin{example}
Consider the 2-category generated by:
\begin{itemize}
\item a 0-cell $x$,
\item a 1-cell $f : x \to x$, and
\item a 2-cell $m : f \circ f \to f$ (i.e.\ a monoid).
\end{itemize}
The 2-cell $m$ is represented by the following 2-diagram:
\begin{center}
\input{figures/monoid.tikz}
\end{center}
where $\Sigma = \set{x, f, m}$ with $\dim x = 0$, $\dim f = 1$, and $\dim m = 2$.
\end{example}

\begin{example}
Consider the 3-category generated by:
\begin{itemize}
\item a 0-cell $x$, and
\item a pair of 2-cells $\alpha, \beta : \id x \to \id x$ (i.e.\ two scalars).
\end{itemize}
The \emph{Eckmann-Hilton move} is the 3-cell represented by the following 3-diagram, which can be interpreted as ``braiding $\alpha$ over $\beta$'':
\begin{center}
\input{figures/braiding.tikz}
\end{center}
where $\Sigma = \set{x, \alpha, \beta}$ with $\dim x = 0$, $\dim \alpha = 2$, and $\dim \beta = 2$.
\end{example}

\subsection{Explosion}

Given a diagram in the zigzag category $\Zig(\cat C)$, we can ``explode'' it to obtain a diagram in $\cat C$.
This can be understood as the left adjoint of the zigzag construction, a result which is inspired by the work of Heidemann~\cite{Hei23}, however the presentation  here is original.

\begin{definition}
Given a category $\cat C$ with a functor $F : \cat C \to \Ord$, we define its \emph{explosion} $\Exp_F(\cat C)$ to be the category consisting of:
\begin{itemize}
\item for every $x$ in the fibre of $n$, objects
\begin{align*}
&\sing^x_i &i &\in \ord n \\
&\reg^x_j &j &\in \ord{n + 1}
\end{align*}
\item for every $f : x \to y$ in the fibre of $\alpha : \ord n \to \ord m$, morphisms
\begin{align*}
\sing^x_i &\xrightarrow{\sing^f_i} \sing^y_{\alpha(i)} &i &\in \ord n \\
\reg^x_{\Reg \alpha(j)} &\xrightarrow{\reg^f_j} \reg^y_j &j &\in \ord{m + 1} \\
\reg^x_j &\xrightarrow{\delta^f_{i, j}} \sing^y_i &i &\in \ord m, \, j \in \Reg \alpha([i, i + 1])
\end{align*}
\end{itemize}
\end{definition}

\begin{figure}
\centering
\input{figures/explosion.tikz}
\caption{A diagram $F : \ord 2 \to \Ord$ together with its explosion $\Exp_F(\ord 2)$.}
\end{figure}

\begin{proposition}
We have an adjunction
\[
\Exp : \Cat / \Ord \rightleftarrows \Cat : \Zig
\]
\end{proposition}

\begin{proof}
For every $\cat J \to \Ord$ and $\cat C$, we define a bijection
\begin{prooftree}
\AxiomC{$F : \cat J \to \Zig(\cat C)$}
\UnaryInfC{$\smallint F : \Exp(\cat J) \to \cat C$}
\end{prooftree}
Given a diagram $F$, we define the diagram $\smallint F$ as follows:
\begin{align*}
\smallint F(\sing^x_i) &\coloneqq F(x)(\sing i) \\
\smallint F(\reg^x_j) &\coloneqq F(x)(\reg j) \\
\smallint F(\sing^f_i) &\coloneqq F(f)(\sing i) \\
\smallint F(\reg^f_j) &\coloneqq F(f)(\reg j) \\
\smallint F(\delta^f_{i, j}) &\coloneqq F(f)(\reg j, \sing i)
\end{align*}
It is easy to check this is a natural isomorphism, so $\Exp \dashv \Zig$.
\end{proof}

Explosion can be iterated to also explode diagrams in $\Zig^n(\cat C)$:
\begin{prooftree}
\AxiomC{$F : \cat J \to \Zig^n(\cat C)$}
\UnaryInfC{$\smallint^k F : \Exp^k(\cat J) \to \Zig^{n - k}(\cat C)$}
\end{prooftree}
We call this the \emph{$k$-fold explosion}.
In particular, note that:
\begin{itemize}
\item If $X$ is an $n$-zigzag, seen as a functor $\ord 1 \to \Zig^n(\cat C)$, its $n$-fold explosion $\Exp^n(\ord 1) \to \cat C$ is $X$ seen as a diagram in $\cat C$.
\item If $f$ is an $n$-zigzag map, seen as a functor $\ord 2 \to \Zig^n(\cat C)$, its $n$-fold explosion $\Exp^n(\ord 2) \to \cat C$ is $f$ seen as a diagram in $\cat C$.
\end{itemize}

\subsection{Contraction}

Recall the definition of contraction by Reutter and Vicary~\cite{RV19}.

\begin{definition}
Let $X$ be a zigzag in $\cat C$.
We define its \emph{contraction} to be the zigzag map arising from taking the colimit in $\cat C$, if it exists, of the underlying diagram $\smallint X : \Exp(\ord 1) \to \cat C$, i.e.
\[
\begin{tikzcd}[column sep=small]
& X(\reg 0) \rar & C & X(\reg n) \lar & \\
X(\reg 0) \urar[equal] \rar & X(\sing 0) \urar & \cdots \lar \rar & X(\sing n - 1) \ular & X(\reg n) \lar \ular[equal]
\end{tikzcd}
\]
If the colimit does not exist, then the contraction is not defined.
\end{definition}

Reutter and Vicary have also given a procedure which computes connected colimits in $\Zig_=(\cat C)$.
This allows contraction to work on $n$-diagrams for every $n$.
There are three main steps to this procedure, which can be stated as the following technical lemmas. 

\begin{lemma}
\label{lma:colimit-projection}
If $\cat C$ is a disjoint union of categories, each having a terminal object, then $\pi : \Zig_=(\cat C) \to \Ord$ preserves connected colimits.
\end{lemma}

\begin{proof}
See~\cite[Proposition 39]{RV19}.
\end{proof}

\begin{lemma}
\label{lma:colimit-restriction}
Consider a connected diagram $F : \cat J \to \Zig_=(\cat C)$ and a cocone $\eta : F \Rightarrow \Delta_X$ such that $\pi \circ \eta$ is a colimit in $\Ord$.
Then $\eta$ is a colimit iff for every $i \in \abs{X}$, the following cocone is a colimit:
\[
\eta^i : F^i \Rightarrow \Delta_{X[i, i + 1]} : \cat J \to \Zig_=(\cat C)
\]
with each $\eta^i_x : F^i_x \to X[i, i + 1]$ given by the restriction $\eta_x[i, i + 1]$.
\end{lemma}

\begin{proof}
This follows from~\cite[Theorem 33]{RV19}.
\end{proof}

\begin{lemma}
\label{lma:colimit-explosion}
Consider a connected diagram $F : \cat J \to \Zig_=(\cat C)$ and a cocone $\eta : F \Rightarrow \Delta_X$ such that $\pi \circ \eta$ is a colimit in $\Ord$ and $\abs{X} = 1$.
Then $\eta$ is a colimit iff the following cocone is a colimit in $\cat C$:
\[
\smallint \eta : \smallint F \Rightarrow \Delta_{X(\sing 0)} : \Exp(\cat J) \to \cat C
\]
where the components are given by the singular and diagonal slices:
\begin{align*}
\smallint \eta(\sing^x_i) &\coloneqq F_x(\sing i) \xrightarrow{\eta_x(\sing i) } X(\sing 0) \\
\smallint \eta(\reg^x_j) &\coloneqq F_x(\reg j) \xrightarrow{\eta_x(\reg j, \sing 0)} X(\sing 0)
\end{align*}
\end{lemma}

\begin{proof}
This follows from~\cite[Theorem 33]{RV19}.
\end{proof}

\noindent
Therefore, the colimit of a connected diagram $F : \cat J \to \Zig_=(\cat C)$ is obtained by taking a colimit in $\Ord$, say $n$, and then an $\ord n$-indexed family of colimits in $\cat C$, namely the colimits of the explosions $\smallint F^i$.
If any of these colimits do not exist, neither does the colimit of $F$.

%% file: figures/monoid.tikz
\begin{tikzpicture}

\node (R0R0) at (0, 0) {$x$};
\node (R0S0) at (1, 0) {$f$};
\node (R0R1) at (2, 0) {$x$};
\node (R0S1) at (3, 0) {$f$};
\node (R0R2) at (4, 0) {$x$};

\node (S0R0) at (0, 1) {$x$};
\node (S0S0) at (2, 1) {$m$};
\node (S0R1) at (4, 1) {$x$};

\node (R1R0) at (0, 2) {$x$};
\node (R1S0) at (2, 2) {$f$};
\node (R1R1) at (4, 2) {$x$};

\draw[->] (R0R0) to (R0S0);
\draw[->] (R0R1) to (R0S0);
\draw[->] (R0R1) to (R0S1);
\draw[->] (R0R2) to (R0S1);

\draw[->] (S0R0) to (S0S0);
\draw[->] (S0R1) to (S0S0);

\draw[->] (R1R0) to (R1S0);
\draw[->] (R1R1) to (R1S0);

\draw[->] (R0R0) to (S0R0);
\draw[->] (R0S0) to (S0S0);
\draw[->] (R0S1) to (S0S0);
\draw[->] (R0R2) to (S0R1);

\draw[->] (R1R0) to (S0R0);
\draw[->] (R1S0) to (S0S0);
\draw[->] (R1R1) to (S0R1);

\end{tikzpicture}

%% file: figures/braiding.tikz
\begin{tikzpicture}[scale=0.75]

\node (R0R0R0) at (1, 0) {$x$};

\node (R0S0R0) at (0, 1) {$x$};
\node (R0S0S0) at (1, 1) {$\alpha$};
\node (R0S0R1) at (2, 1) {$x$};

\node (R0R1R0) at (1, 2) {$x$};

\node (R0S1R0) at (0, 3) {$x$};
\node (R0S1S0) at (1, 3) {$\beta$};
\node (R0S1R1) at (2, 3) {$x$};

\node (R0R2R0) at (1, 4) {$x$};

\draw (R0R0R0) to (R0S0R0);
\draw (R0R0R0) to (R0S0R1);

\draw (R0S0R0) to (R0S0S0);
\draw (R0S0R1) to (R0S0S0);

\draw (R0R1R0) to (R0S0R0);
\draw (R0R1R0) to (R0S0R1);

\draw (R0R1R0) to (R0S1R0);
\draw (R0R1R0) to (R0S1R1);

\draw (R0S1R0) to (R0S1S0);
\draw (R0S1R1) to (R0S1S0);

\draw (R0R2R0) to (R0S1R0);
\draw (R0R2R0) to (R0S1R1);

\node (S0R0R0) at (6, 1) {$x$};

\node (S0S0R0) at (4, 2) {$x$};
\node (S0S0S0) at (5, 2) {$\alpha$};
\node (S0S0R1) at (6, 2) {$x$};
\node (S0S0S1) at (7, 2) {$\beta$};
\node (S0S0R2) at (8, 2) {$x$};

\node (S0R1R0) at (6, 3) {$x$};

\draw (S0R0R0) to (S0S0R0);
\draw (S0R0R0) to (S0S0R1);
\draw (S0R0R0) to (S0S0R2);

\draw (S0S0R0) to (S0S0S0);
\draw (S0S0R1) to (S0S0S0);
\draw (S0S0R1) to (S0S0S1);
\draw (S0S0R2) to (S0S0S1);

\draw (S0R1R0) to (S0S0R0);
\draw (S0R1R0) to (S0S0R1);
\draw (S0R1R0) to (S0S0R2);

\node (R1R0R0) at (11, 0) {$x$};

\node (R1S0R0) at (10, 1) {$x$};
\node (R1S0S0) at (11, 1) {$\beta$};
\node (R1S0R1) at (12, 1) {$x$};

\node (R1R1R0) at (11, 2) {$x$};

\node (R1S1R0) at (10, 3) {$x$};
\node (R1S1S0) at (11, 3) {$\alpha$};
\node (R1S1R1) at (12, 3) {$x$};

\node (R1R2R0) at (11, 4) {$x$};

\draw (R1R0R0) to (R1S0R0);
\draw (R1R0R0) to (R1S0R1);

\draw (R1S0R0) to (R1S0S0);
\draw (R1S0R1) to (R1S0S0);

\draw (R1R1R0) to (R1S0R0);
\draw (R1R1R0) to (R1S0R1);

\draw (R1R1R0) to (R1S1R0);
\draw (R1R1R0) to (R1S1R1);

\draw (R1S1R0) to (R1S1S0);
\draw (R1S1R1) to (R1S1S0);

\draw (R1R2R0) to (R1S1R0);
\draw (R1R2R0) to (R1S1R1);

\draw[->] (2.75, 2) to (3.25, 2);
\draw[->] (9.25, 2) to (8.75, 2); 

\end{tikzpicture}

%% file: figures/explosion.tikz
\begin{tikzpicture}

\node (a) at (-1, 0) {$\ord 3$};
\node (b) at (-1, 1) {$\ord 4$};

\draw[->] (a) to (b);

\node (a_r0) at (2, 0) {$\bullet$};
\node (a_s0) at (3, 0) {$\bullet$};
\node (a_r1) at (4, 0) {$\bullet$};
\node (a_s1) at (5, 0) {$\bullet$};
\node (a_r2) at (6, 0) {$\bullet$};
\node (a_s2) at (7, 0) {$\bullet$};
\node (a_r3) at (8, 0) {$\bullet$};

\node at (0, 0.5) {$\mapsto$};

\draw[->] (a_r0) to (a_s0);
\draw[->] (a_r1) to (a_s0);
\draw[->] (a_r1) to (a_s1);
\draw[->] (a_r2) to (a_s1);
\draw[->] (a_r2) to (a_s2);
\draw[->] (a_r3) to (a_s2);

\node (b_r0) at (1, 1) {$\bullet$};
\node (b_s0) at (2, 1) {$\bullet$};
\node (b_r1) at (3, 1) {$\bullet$};
\node (b_s1) at (4, 1) {$\bullet$};
\node (b_r2) at (5, 1) {$\bullet$};
\node (b_s2) at (6, 1) {$\bullet$};
\node (b_r3) at (7, 1) {$\bullet$};
\node (b_s3) at (8, 1) {$\bullet$};
\node (b_r4) at (9, 1) {$\bullet$};

\draw[->] (b_r0) to (b_s0);
\draw[->] (b_r1) to (b_s0);
\draw[->] (b_r1) to (b_s1);
\draw[->] (b_r2) to (b_s1);
\draw[->] (b_r2) to (b_s2);
\draw[->] (b_r3) to (b_s2);
\draw[->] (b_r3) to (b_s3);
\draw[->] (b_r4) to (b_s3);

\draw[->] (a_r0) to (b_r0);
\draw[->] (a_s0) to (b_s0);
\draw[->] (a_r1) to (b_r1);
\draw[->] (a_r1) to (b_r2);
\draw[->] (a_s1) to (b_s2);
\draw[->] (a_s2) to (b_s2);
\draw[->] (a_r3) to (b_r3);
\draw[->] (a_r3) to (b_r4);

\end{tikzpicture}

%% file: 4-anticontraction.tex
\section{Anticontraction} \label{sec:anticontraction}

We  now define the dual of contraction, called \emph{anticontraction}.
Just as contraction is based on colimits, this is based on anticolimits.

\begin{definition}
For $n \in \mathbb{N}$, define the following hypergraph:
\[
V \coloneqq \ord{n + 1} \qquad E \coloneqq \ord n \qquad i(e) = \set{e, e + 1}
\]
Let $\pos J_n$ be the associated poset, which we can sketch as follows:
\[
\begin{tikzcd}[column sep=small]
\bullet && \cdots && \bullet \\
& \bullet \ular \urar && \bullet \ular \urar &
\end{tikzcd}
\]
\end{definition}

\begin{definition}
Let $X$ be a zigzag map of length 1 in $\cat C$ together with a sink in $\cat C$ into the unique singular object of $X$:
\[
\set{A_i \to X(\sing 0)}_{i \in \ord{n + 1}}
\]
An \emph{anticontraction} of $X$ is a zigzag map whose singular slices are $\set{A_i \to X(\sing 0)}$ such that it is a contraction of its source, i.e.\
\[
\begin{tikzcd}[column sep=small]
& X(\reg 0) \rar & X(\sing 0) & X(\reg 1) \lar & \\
X(\reg 0) \rar \urar[equal] & A_0 \urar & \cdots \lar \rar & A_n \ular & X(\reg 1) \lar \ular[equal]
\end{tikzcd}
\]
In particular, an anticontraction consists of the following data:
\begin{itemize}
\item a $\pos J_n$-anticolimit of the sink $\set{A_i \to X(\sing 0)}$,
\item a lift of $X(\reg 0) \to X(\sing 0)$ along the morphism $A_0 \to X(\sing 0)$,
\item a lift of $X(\reg 1) \to X(\sing 0)$ along the morphism $A_n \to X(\sing 0)$.
\end{itemize}
\end{definition}

\subsection{Zigzag anticolimits}

We give a procedure for computing anticolimits in $\Zig_=(\cat C)$, which is analogous to the procedure for colimits by Reutter and Vicary.

\begin{theorem}
Let $\cat C$ be a disjoint union of categories, each having a terminal object.
Then for every sink of the following form
\[
\kappa : X \Rightarrow \Delta_Y : \max \pos J \to \Zig_=(\cat C)
\]
for a connected poset $\pos J$, we have an embedding of categories
\[
\Acl_\pos J(\kappa) \hookrightarrow \coprod\nolimits_{F \in \Acl_\pos J(\pi \circ \kappa)} \prod\nolimits_{i \in \abs{Y}} \Acl_{\Exp_{F^i}(\pos J)}(\smallint \kappa^i)
\]
i.e.\ $\pos J$-anticolimits of $\kappa$ are uniquely determined by:
\begin{itemize}
\item a $\pos J$-anticolimit $F$ of $\pi \circ \kappa$ in $\Ord$, and
\item for every $i \in \abs{Y}$, a $\Exp_{F^i}(\pos J)$-anticolimit of $\smallint \kappa^i$ in $\cat C$.
\end{itemize}
Note that $\smallint \kappa^i$ is defined as in \Cref{lma:colimit-restriction,lma:colimit-explosion}, and $F^i : \pos J \to \Ord$ is the restriction of $F : \pos J \to \Ord$ to the preimages of $i$ under $\pi \circ \kappa$.
\end{theorem}

\begin{proof}
Let $A$ be an anticolimit of $\kappa$ in $\Zig_=(\cat C)$.
Then $F = \pi \circ A$ is an anticolimit of $\pi \circ \kappa$ in $\Ord$ by \Cref{lma:colimit-projection}.
Moreover, for every $i \in \abs{Y}$, the explosion of the restriction $\smallint A^i : \Exp_{F^i}(\pos J) \to \cat C$ is an anticolimit of $\smallint \kappa^i$ in $\cat C$ by \Cref{lma:colimit-restriction,lma:colimit-explosion}.
Hence, the map
\[
A \mapsto (F, (\smallint A^i)_{i \in \abs{Y}})
\]
is a well-defined map of the desired type.
The map is injective, as if $A \neq B$ then either $\pi \circ A \neq \pi \circ B$ or $\smallint A^i \neq \smallint B^i$ for some $i$.
It is easy to see that this extends to a functor, and hence an embedding.
\end{proof}

While the embedding may fail to be epic, we can detect when something is in its image.
Suppose  we are given the following:
\[
F \in \Acl_\pos J(\pi \circ \kappa) \qquad G_i \in \Acl_{\Exp_{F^i}(\pos J)}(\smallint \kappa^i)
\]
We can attempt to construct $A \in \Acl_\pos J(\kappa)$ as follows:
\begin{enumerate}
\item Unexplode each $G_i$ to get a diagram $A_i : \pos J \to \Zig(\cat C)$.
\item Check that $A_i$ is globular, otherwise fail.
\item Then $\kappa^i$ is a colimit of $A_i$ by \Cref{lma:colimit-explosion}, so by globularity, every zigzag in $A_i$ has the same first and last regular objects as $Y[i, i + 1]$, i.e.\ $Y(\reg i)$ and $Y(\reg i + 1)$ respectively.
Therefore, we can concatenate the $A_i$ together to get a diagram
\[
A : \pos J \to \Zig_=(\cat C)
\]
with $A^i = A_i$. Then by \Cref{lma:colimit-restriction}, $\kappa$ is a colimit of $A$.
\end{enumerate}

Therefore, anticolimits in $\Zig_=(\cat C)$ can be computed in terms of anticolimits in $\Ord$ and $\cat C$.
Since we know how to do anticolimits in $\Ord$, as shown in section 2, we can extract a recursive procedure for anticolimits in $\Zig^n_=(\cat C)$ in terms of anticolimits in $\cat C$, for all $n$.

\subsection{Zigzag factorisation}

We now show that if a category $\cat C$ admits a factorisation structure for sinks, we can lift it to the globular zigzag category $\Zig_=(\cat C)$.
First, recall the definition of a factorisation structure for sinks, according to~\cite[Section 3]{Cas12}, as well as an useful technical result.

\begin{definition}
Let $\mathbf{E}$ be a conglomerate of sinks in $\cat C$, and let $M$ be a class of morphisms in $\cat C$, both closed under composition with isomorphisms.
Then $(\mathbf{E}, M)$ is a \emph{factorisation structure} on $\cat C$ if:
\begin{itemize}
\item Every sink $\set{f_i : A_i \to B}_{i \in I}$ in $\cat C$ factors as $f_i = m \circ e_i$ for a sink $\set{e_i : A_i \to C}_{i \in I}$ in $\mathbf{E}$ and a morphism $m : C \to B$ in~$M$.
\item ($\mathbf{E}, M)$ has the \emph{orthogonal lifting property}, i.e.\ for every sink $\set{e_i : A_i \to B}_{i \in I}$ in $\mathbf{E}$ and morphism $m : C \to D$ in $M$, if the outer square commutes for every $i$, there is a unique lift:
\[
\begin{tikzcd}
A_i \dar["e_i"'] \rar & C \dar["m"] \\
B \rar \urar[dashed, "\exists!" description] & D
\end{tikzcd}
\]
\end{itemize}
\end{definition}

\begin{lemma}
\label{lma:factorisation-mono}
If $(\mathbf{E}, M)$ is a factorisation structure, then $M \subseteq \mathrm{Mono}$.
\end{lemma}

\begin{proof}
See~\cite[Proposition 3.6]{Cas12}.
\end{proof}

\begin{definition}
If $\mathbf{E}$ is a conglomerate of sinks and $M$ is a class of morphisms in $\cat C$, then we define the following:
\begin{itemize}
\item a sink $\set{f_i : X_i \to Y}_{i \in I}$ in $\Zig_=(\cat C)$ is \emph{$\mathbf{E}$-singular} if for every $k \in \abs{Y}$, the sink of all singular slices into $Y(\sing k)$ is in $\mathbf{E}$:
\[
\set{f_i(\sing j) : X_i(\sing j) \to Y(\sing k) \mid i \in I, \, j \in (f_i)_\mathsf{s}^{-1}(k)}
\]
\item a zigzag map is an \emph{$M$-relabelling} if it is globular, its singular map is the identity, and all its singular slices are in $M$.
\end{itemize}
\end{definition}

We write $\Sing(\mathbf{E})$ for the conglomerate of $\mathbf{E}$-singular sinks and $\Relab(M)$ for the class of $M$-relabellings.
Note that these are both closed under composition with isomorphisms, if $\mathbf{E}$ and $M$ are.

\begin{theorem}
If $(\mathbf{E}, M)$ is a factorisation structure on $\cat C$, then we have a factorisation structure $(\Sing(\mathbf{E}), \Relab(M))$ on $\Zig_=(\cat C)$.
\end{theorem}

\begin{proof}
Given a sink of globular zigzag maps $\set{f_i : X_i \to Y}_{i \in I}$, we want to define an $\mathbf{E}$-singular sink and an $M$-relabelling
\[
\set{e_i : X_i \to Z}_{i \in I} \qquad m : Z \to Y
\]
such that $f_i = m \circ e_i$ for all $i \in I$.
In particular, we must have that
\begin{align*}
\abs{Z} &\coloneqq \abs{Y} \\
Z(\reg j) &\coloneqq Y(\reg j) \\
(e_i)_\mathsf{s} &\coloneqq (f_i)_\mathsf{s} \\
m_\mathsf{s} &\coloneqq \id
\end{align*}
For every $k \in \abs{Y}$, construct the sink of singular slices into $Y(\sing k)$:
\[
\set{f_i(\sing j) : X_i(\sing j) \to Y(\sing k) \mid i \in I, \, j \in (f_i)_\mathsf{s}^{-1}(k)} \\
\]
If this is empty, define $Z(\sing k) \coloneqq Y(\sing k)$ and $m(\sing k) \coloneqq \id$.
Otherwise, we can factor this sink in $\cat C$ as an $\mathbf{E}$-sink and an $M$-morphism:
\[
\set{e_i(\sing j) : X_i(\sing j) \to Z(\sing k) \mid i \in I, \, j \in (f_i)_\mathsf{s}^{-1}(k)} \\
\]
\[
m(\sing k) : Z(\sing k) \to Y(\sing k)
\]
All that is left is to define the cospan maps
\[
Z(\reg k) \to Z(\sing k) \qquad Z(\reg k + 1) \to Z(\sing k)
\]
such that they satisfy the following commutativity conditions:
\begin{equation}
\begin{tikzcd}[column sep=small]
Y(\reg k) \rar & Y(\sing k) \\
Z(\reg k) \rar \uar[equal] & Z(\sing k) \uar["m(\sing k)"']
\end{tikzcd}
\qquad
\begin{tikzcd}[column sep=small]
Y(\sing k) & Y(\reg k + 1) \lar \\
Z(\sing k) \uar["m(\sing k)"] & Z(\reg k + 1) \lar \uar[equal]
\end{tikzcd}
\end{equation}
If $(f_i)_\mathsf{s}^{-1}(k)$ is empty:
\begin{equation}
\begin{tikzcd}[column sep=small]
Z(\reg k) \rar & Z(\sing k) & Z(\reg k + 1) \lar \\
& X_i(\reg (f_i)_\mathsf{r}(k)) \ular[equal] \urar[equal] &
\end{tikzcd}
\end{equation}
If $(f_i)_\mathsf{s}^{-1}(k) = [a, b]$ is non-empty:
\begin{equation}
\begin{tikzcd}[column sep=small]
Z(\reg k) \rar & Z(\sing k) \\
X_i(\reg a) \rar \uar[equal] & X(\sing a) \uar["e_i(\sing a)"']
\end{tikzcd}
\qquad
\begin{tikzcd}[column sep=small]
Z(\sing k) & Z(\reg k + 1) \lar \\
X_i(\sing b) \uar["e_i(\sing b)"] & X_i(\reg b + 1) \lar \uar[equal]
\end{tikzcd}
\end{equation}

\begin{equation}
\begin{tikzcd}[column sep=small]
& Z(\sing k) & \\
X_i(\sing j) \urar["e_i(\sing j)"] & X_i(\reg j + 1) \lar \rar & X_i(\sing j + 1) \ular["e_i(\sing j + 1)"']
\end{tikzcd}
\end{equation}
If all preimages are empty, then $m(\sing k) = \id$, so we can define
\begin{align*}
Z(\reg k) \to Z(\sing k) &\coloneqq Y(\reg k) \to Y(\sing k) \\
Z(\reg k + 1) \to Z(\sing k) &\coloneqq Y(\reg k + 1) \to Y(\sing k)
\end{align*}
such that (1) and (2) hold.
Otherwise, if at least one of the preimages is non-empty, say $(f_i)_\mathsf{s}^{-1}(k) = [a, b]$, then we can define
\begin{align*}
Z(\reg k) \to Z(\sing k) &\coloneqq X_i(\reg a) \to X(\sing a) \xrightarrow{e_i(\sing a)} Z(\sing k) \\
Z(\reg k + 1) \to Z(\sing k) &\coloneqq X_i(\reg b + 1) \to X(\sing b) \xrightarrow{e_i(\sing b)} Z(\sing k)
\end{align*}
Now (1) always holds, and (2), (3), and (4) hold when post-composed with $m(\sing k)$, which is monic by \Cref{lma:factorisation-mono}, so they hold as well.

Finally, suppose that we have a family of commutative squares 
\[
\begin{tikzcd}
X_i \dar["e_i"'] \rar["f_i"] & Z \dar["m"] \\
Y \rar["g"'] & W
\end{tikzcd}    
\]
with $\set{e_i} \in \Sing(\mathbf{E})$ and $m \in \Relab(M)$.
We want to define a lift:
\[
\begin{tikzcd}
X_i \dar["e_i"'] \rar["f_i"] & Z \dar["m"] \\
Y \rar["g"'] \urar[dashed, "h" description] & W
\end{tikzcd}    
\]
Since $m$ is an $M$-relabelling, we necessarily have that $h_\mathsf{s} = g_\mathsf{s}$.
Now for every $k \in \abs{Y}$ and $j \in (g_\mathsf{s})^{-1}(k)$, the following square commutes, so by the orthogonal lifting property of $(\mathbf{E}, M)$, it has a unique lift:
\[
\begin{tikzcd}
X_i(\sing j) \dar["e_i(\sing j)"'] \rar["f_i(\sing j)"] & Z(\sing g_\mathsf{s}(k)) \dar["m(\sing g_\mathsf{s}(k))"] \\
Y(\sing k) \rar["g(\sing k)"'] \urar[dashed, "h(\sing k)" description] & W(\sing g_\mathsf{s}(k))
\end{tikzcd}    
\]
Therefore, $h$ is uniquely determined, concluding this proof.
\end{proof}

This factorisation enables us to propagate an anticontraction from a singular object of a zigzag to the whole zigzag. 
This is called \emph{recursive anticontraction} and is a generalisation of the recursive expansion procedure described in~\mbox{\cite[Section 5]{RV19}.}

\paragraph{Recursive anticontraction}
Suppose we are performing a recursive anticontraction on the singular object of a zigzag, so we have 
\[
\begin{tikzcd}[column sep=small]
X(\reg 0) \rar & X(\sing 0) & X(\reg 1) \lar \\
& A \uar &
\end{tikzcd}
\]
We lift this to a zigzag map according to the following scheme:
\begin{enumerate}
\item[(a)] First, attempt to find lifts on both sides as follows:
\[
\begin{tikzcd}[column sep=small]
X(\reg 0) \rar & X(\sing 0) & X(\reg 1) \lar \\
X(\reg 0) \uar[equal] \rar[dashed] & A \uar & X(\reg 1) \lar[dashed] \uar[equal]
\end{tikzcd}
\]
\item[(b)] If one lift exists, say $X(\reg 0) \to A$ but not the other, factorise the other morphism as $X(\reg 1) \to B \rightarrowtail X(\sing 0)$ using the factorisation structure on zigzag maps, and then attempt to find the anticontraction:
\[
\begin{tikzcd}[column sep=small]
& X(\reg 0) \rar & X(\sing 0) & X(\reg 1) \lar & \\
X(\reg 0) \urar[equal] \rar & A \urar & C \lar[dashed] \rar[dashed] & B \ular & X(\reg 1) \lar \ular[equal]
\end{tikzcd}
\]
where $X(\reg 1) \to X(\sing 0)$ factors as $X(\reg 1) \to B \rightarrowtail X(\sing 0)$.
\item[(c)] As a fallback, if neither lift exists, or no anticontraction in (2) exists, then just introduce a ``bubble'' as follows, re-using the given map $A \to X(\sing 0)$:
\[
\begin{tikzcd}[column sep=small]
& X(\reg 0) \rar & X(\sing 0) & X(\reg 1) \lar & \\
X(\reg 0) \urar[equal] \rar & X(\sing 0) \urar[equal] & A \lar \rar & X(\sing 0) \ular[equal] & X(\reg 1) \lar \ular[equal]
\end{tikzcd}
\]
\end{enumerate}

\noindent
Step (b) here is novel, and introduces new capabilities to the proof assistant.
We provide the following example.
This is a zigzag map obtained by step (b): the first singular slice is a relabelling given by factorisation, and the second singular slice is another anticontraction.
Note that this is step 2 of the naturality move in \Cref{fig:naturality-move-v2}.
\[
\begin{aligned}\scalebox{0.5}{\input{figures/naturality_v2/regular1.tikz}}\end{aligned}
\quad
\begin{aligned}
\begin{tikzpicture}
\node (R0) at (0, 0.5) {};
\node (S0) at (0, 1.5) {};
\node (R1) at (0, 2.5) {};

\node (R0') at (1, 0) {};
\node (S0') at (1, 1) {};
\node (S1') at (1, 2) {};
\node (R2') at (1, 3) {};

\draw[->] (S0') to (S0);
\draw[->] (S1') to (S0);

\draw[double, double distance=2pt] (R0') to (R0);
\draw[double, double distance=2pt] (R2') to (R1);
\end{tikzpicture}
\end{aligned}
\quad
\begin{aligned}\scalebox{0.5}{\input{figures/naturality_v2/regular2.tikz}}\end{aligned}
\]
Similarly, steps 5--6 of the Reidemeister III construction in \Cref{fig:reidemeister-move-v2} are also recursive anticontractions obtained by step (b).

%% file: figures/naturality_v2/regular1.tikz
\begin{tikzpicture}
\definecolor{generator-1-2-0-pos}{RGB}{192, 57, 43}
\definecolor{generator-2-2-0-pos}{RGB}{243, 156, 18}
\definecolor{generator-3-3-0-pos}{RGB}{142, 68, 173}
\definecolor{generator-0-0-1-pos}{RGB}{246, 245, 244}

\newcommand{\wire}[2]{
  \ifdefined\recolor\draw[color=\recolor, line width=10pt]\else\draw[color=#1, line width=5pt]\fi #2;
}
\newcommand{\clipped}[3]{
\begin{scope}
  \newcommand{\recolor}{#1}
  \clip#3;
  #2
\end{scope}
}

\begin{scope}[transparency group]
% Background surfaces
\fill[generator-0-0-1-pos] (0,0) -- (6,0) -- (6,4) -- (0,4) -- (0,0);
\newcommand{\layer}[1]{
  \clipped{generator-0-0-1-pos}{#1}{(0,0) -- (6,0) -- (6,4) -- (0,4) -- (0,0)}
  #1
}

% Wire layers
\wire{generator-2-2-0-pos}{(4,1) .. controls (4,1.8) and (3.6,2) .. (3,2) .. controls (2.4,2) and (2,2.2) .. (2,3)};
\layer{
\wire{generator-1-2-0-pos}{(2,0) -- (2,1) .. controls (2,1.8) and (2.4,2) .. (3,2) .. controls (3.6,2) and (4,2.2) .. (4,3) -- (4,4)};
\wire{generator-2-2-0-pos}{(4,0) -- (4,1)(2,3) -- (2,4)};
}
\end{scope}
\fill[generator-3-3-0-pos] (3,2) circle (0.21);
\end{tikzpicture}

%% file: figures/naturality_v2/regular2.tikz
\begin{tikzpicture}
\definecolor{generator-1-2-0-pos}{RGB}{192, 57, 43}
\definecolor{generator-2-2-0-pos}{RGB}{243, 156, 18}
\definecolor{generator-3-3-0-pos}{RGB}{142, 68, 173}
\definecolor{generator-0-0-1-pos}{RGB}{246, 245, 244}

\newcommand{\wire}[2]{
  \ifdefined\recolor\draw[color=\recolor, line width=10pt]\else\draw[color=#1, line width=5pt]\fi #2;
}
\newcommand{\clipped}[3]{
\begin{scope}
  \newcommand{\recolor}{#1}
  \clip#3;
  #2
\end{scope}
}

\begin{scope}[transparency group]
% Background surfaces
\fill[generator-0-0-1-pos] (0,0) -- (6,0) -- (6,6) -- (0,6) -- (0,0);
\newcommand{\layer}[1]{
  \clipped{generator-0-0-1-pos}{#1}{(0,0) -- (6,0) -- (6,6) -- (0,6) -- (0,0)}
  #1
}

% Wire layers
\wire{generator-2-2-0-pos}{(4,1) .. controls (4,1.8) and (3.6,2) .. (3,2) .. controls (2.4,2) and (2,2.2) .. (2,3)};
\layer{
\wire{generator-1-2-0-pos}{(2,0) -- (2,1) .. controls (2,1.8) and (2.4,2) .. (3,2) .. controls (3.6,2) and (4,2.2) .. (4,3) -- (4,6)};
\wire{generator-2-2-0-pos}{(4,0) -- (4,1)(2,3) -- (2,6)};
}
\end{scope}
\fill[generator-3-3-0-pos] (4,4) circle (0.21);
\end{tikzpicture}